\documentclass[11pt]{article}
\usepackage{}
\usepackage{mathbbold}
\usepackage{bbm}
\usepackage{amsfonts}
\usepackage[T1]{fontenc}
\usepackage{latexsym,amssymb,amsmath,amsfonts,amsthm}
\usepackage{graphicx}
\usepackage{subfigure}
\usepackage{epsf}
\usepackage{epstopdf}
\usepackage{amssymb}

\renewcommand{\theequation}{\arabic{section}.\arabic{equation}}

\newcommand{\R}{{\mathbb R}}
\newcommand{\Z}{{\mathbb Z}}

\newcommand{\C}{{\mathbb C}}

\newcommand{\be}{\begin{eqnarray}}
\newcommand{\ben}{\begin{eqnarray*}}
\newcommand{\en}{\end{eqnarray}}
\newcommand{\enn}{\end{eqnarray*}}
\newcommand{\ba}{\backslash}
\newcommand{\pa}{\partial}

\newcommand{\ov}{\overline}

\newcommand{\G}{\Gamma}
\newcommand{\eps}{\epsilon}

\newcommand{\om}{\omega}

\newcommand{\wi}{\widetilde}

\newcommand{\hth}{\hat{\theta}}
\newcommand{\hx}{\hat{x}}
\newtheorem{theorem}{Theorem}[section]
\newtheorem{lemma}[theorem]{Lemma}

\begin{document}
\title{\bf Phaseless inverse source scattering problem: phase retrieval, uniqueness and direct sampling methods}
\author{Xia Ji\thanks{LSEC, Academy of Mathematics and Systems Science, Chinese Academy of Sciences, Beijing 100190, China. Email: jixia@lsec.cc.ac.cn (XJ)},
       \and Xiaodong Liu\thanks{Institute of Applied Mathematics, Academy of Mathematics and Systems Science, Chinese Academy of Sciences, 100190 Beijing, China. Email: xdliu@amt.ac.cn (XL)},
       \and Bo Zhang\thanks{LSEC, NCMIS and Academy of Mathematics and Systems Science, Chinese Academy of Sciences,
Beijing 100190, China and School of Mathematical Sciences, University of Chinese Academy of Sciences,
Beijing 100049, China. Email: b.zhang@amt.ac.cn (BZ)}}
\date{}
\maketitle

\begin{abstract}
Similar to the obstacle or medium scattering problems, an important property of the phaseless far field patterns for source scattering problems is the translation invariance.
Thus it is impossible to reconstruct the location of the underlying sources.
Furthermore, the phaseless far field pattern is also invariant if the source is multiplied by any complex number with modulus one. Therefore, the source can not be uniquely determined, even the multi-frequency phaseless far field patterns are considered.
By adding a reference point source into the model, we propose a simple and stable phase retrieval method and
establish several uniqueness results with phaseless far field data.
We proceed to introduce a novel direct sampling method for shape and location reconstruction of the source by using broadband sparse phaseless data directly.
We also propose a combination method with the novel phase retrieval algorithm and the classical direct sampling methods with phased data.
Numerical examples in two dimensions are also presented to demonstrate their feasibility and effectiveness.

\vspace{.2in} {\bf Keywords:}
Phaseless data; phase retrieval; uniqueness; sampling methods; far field pattern;

\vspace{.2in} {\bf AMS subject classifications:}
35P25, 45Q05, 78A46, 74B05

\end{abstract}

\section{Introduction}
Acoustic source imaging problems play an important role in such diverse areas as antenna synthesis, biomedical imaging, sound source localization, or identification of pollutant in the environment. In the last forty years, the inverse acoustic source scattering problems have attracted more and more attention, and significant progress has
been made on uniqueness \cite{AlaHuLiuSun,BaoLiLinTriki,BleisteinCohen,DevaneyMarengoLi,DevaneySherman,ElbadiaNara,EllerValdivia,Griesmaier,SylvesterKelly}, stability analyses \cite{BaoLiLinTriki,BaoLinTriki,ChengIsakovLu,ElbadiaNara,EllerValdivia,IsakovLu} and numerical approaches \cite{AlaHuLiuSun,BaoLiLinTriki,BaoLinTriki,BaoLuRundellXu,ElbadiaNara,EllerValdivia,Griesmaier,WangGuoZhangLiu2017IP, ZhangGuo}.

All the above works consider the case with phased measurements, where the corresponding inverse problems are linear and higher wave number information may yield increased stability \cite{BaoLiLinTriki,ChengIsakovLu,IsakovLu}. However, in many cases of practical interest, it is very difficult and expensive to obtain the phased data, while the phaseless data is much easier to be achieved. The phaseless inverse source problems become nonlinear and the source function can not be uniquely determined, even multi-frequency data are used.
We refer to \cite{BaoLinTriki2011} for a continuation method for source reconstruction with the multi-frequency phaseless Cauchy measurements.
The accuracy is indeed not comparable to those with phased data, which is reasonable because the phaseless inverse problem is nonlinear.
In a recent work \cite{ZhangGuoLiLiu}, the authors introduce a stable phase retrieval method by adding twenty reference point sources with specially arranged locations into the scattering system.
In this paper, following the idea introduced in our recent work \cite{JiLiuZhang} for phaseless inverse obstacle and medium scattering problems, we introduce a novel phase retrieval
technique by using at most three reference point sources. Some uniqueness results and direct sampling methods are also proposed.
We focus on phaseless far field data, and the case with phaseless scattered fields can be done similarly.

We begin with the formulations of the acoustic source scattering problems.
Let $k=\om/c>0$ be the wave number of a time harmonic wave, where $\om>0$ and $c>0$ denote the frequency and sound speed, respectively.
Fixing two wave numbers $0<k_{min}<k_{max}$, we consider the wave equation with
\be\label{kassumption}
k\in(k_{min}, k_{max}).
\en
Let
\ben
D:= \bigcup_{m=1}^{M}D_{m}\subseteq  \mathbb{R}^{n}
\enn
be an ensemble of finitely many well-separated bounded domains in $ \R^{n}$, $n=2, 3,$ i.e., $\overline{D_{j}}\cap\overline{D_{l}}\ =\emptyset$ for $j\neq l$.
For any fixed $k\in(k_{min},k_{max})$, let $S(\cdot,k)\in L^{2}(\R^n)$ represent the acoustic source with compact support $D$.
Then the  time-harmonic wave  $u_{S}\in H_{loc}^{1}( \mathbb{R}^{n})$ radiated by $S$ solves the Helmholtz equation
\be\label{Helmholtzequation}
\Delta u_{S}(x,k) +k^{2} u_{S}(x,k)= S(x,k) \quad \quad in \; \;  \mathbb{R}^{n}
\en
and satisfies the Sommerfeld radiation condition
\be\label{SRC}
\lim_{r\longrightarrow \infty} r^{\frac{n-1}{2}}\Big(\frac{\partial u_{S}}{\partial r} - iku_{S} \Big) = 0, \quad \quad r=|x|.
\en
From the Sommerfeld radiation condition \eqref{SRC}, it is well known that the scattered field $u_S$ has the following asymptotic behavior
\ben
u_{S}(x, k)= C(k,n)\frac{e^{ik|x|}}{|x|^{\frac{n-1}{2}}} u^{\infty}_{S}(\hx, k)+ {\cal O}(|x|^{-\frac{n+1}{2}}), \quad \hx=\frac{x}{|x|}\in \mathbb{S}^{n-1},\;
\enn
as $|x|\longrightarrow \infty,$ where $C(k,n)=e^{i\pi/4}/\sqrt{8\pi k}$ \;if $n=2$ and $C(k,n)=1/4 \pi$ if $n=3.$
The complex valued function $u^\infty_{S}=u^\infty_{S}(\hx, k)$ defined on the unit sphere $\mathbb{S}^{n-1}$
is known as the far field pattern with $\hx\in \mathbb{S}^{n-1}$ denoting the observation direction.

The radiating solution $u_S$ to the scattering problem \eqref{Helmholtzequation}-\eqref{SRC} takes the form
\be\label{us}
u_{S}(x,k)=\int_{\R^n}\Phi_{k}(x,y)S(y,k)ds(y),\quad x\in\R^{n},
\en
with
\ben\label{Phi}
\Phi_{k}(x,y):=\left\{
              \begin{array}{ll}
                \frac{i}{4}H^{(1)}_0(k|x-y|), & n=2; \\
                \frac{ik}{4\pi}h^{(1)}_0(k|x-y|)=\frac{e^{ik|x-y|}}{4\pi|x-y|}, & n=3,
              \end{array}
            \right.
\enn
being the fundamental solution of the Helmholtz equation.
Here, $H^{(1)}_0$ and $h^{(1)}_0$ are, respectively, Hankel function and spherical Hankel function of the first kind and order zero.
From asymptotic behavior of the Hankel functions, we deduce that the corresponding far field pattern has the form
\be\label{uinf}
u^{\infty}_{S}(\hx,k)=\int_{\R^n}e^{-ik\hx\cdot\,y}S(y,k)dy,\quad \hx\in\,\mathbb{S}^{n-1}.
\en
Both the scattered fields $u_{S}$ in \eqref{us} and the corresponding far field patterns $u^{\infty}_{S}$ given in \eqref{uinf} are complex valued functions.
In many practical applications, only the intensity information of these data are available. Thus the corresponding inverse problems are described as follows:\\

{\bf (IP1)}:{\quad\em Determine the source $S$ from one of the following data sets
\ben
{\rm (1).}\Big\{|u_{S}(x, k)|:\,x\in \G,\,k\in (k_{min},k_{max})\Big\}; \quad {\rm (2).}\Big\{|u^\infty_{S}(\hx, k)|:\,\hx\in \mathbb{S}^{n-1},\,k\in (k_{min},k_{max})\Big\},
\enn
where $\G$ is the measurement surface that contains $D$ in its interior.
}\\

Clearly, the phaseless inverse problem {\bf (IP1)} is nonlinear.
For any fixed $\hth\in \mathbb{S}^{n-1}$, define $S_{\hth}(y,k):=e^{i\hth}S(y,k)$. From the representations \eqref{us} and \eqref{uinf}, we deduce the corresponding scattered field $u_{S_{\hth}}$ and far field pattern $u^{\infty}_{S_{\hth}}$ satisfy
\be\label{rotationinvariance}
u_{S_{\hth}}(x, k)=e^{i\hth}u_{S}(x, k),\, x\in\G\quad\mbox{and}\quad  u^{\infty}_{S_{\hth}}(\hx,k)=e^{i\hth}u^{\infty}_{S}(\hx,k),\,\hx\in \mathbb{S}^{n-1},\,k\in (k_{min},k_{max}).
\en
This implies that the phaseless data are invariant under rotation. Thus the source function $S$ can not be uniquely determined from the above phaseless data. For the case with phaseless far field data, even the location of the support of the source can not be uniquely determined. Actually, for any fixed vector $h\in \R^{n}$, define $S_h(y,k):=S(y-h,k)$. Then the corresponding support is given by $D_{h}:=\{x+h:\;x\in D\}$, which is the translation of $D$ with respect to the vector $h\in \R^{n}$.
Then, from \eqref{uinf}, the corresponding far field pattern is given by
\be\label{translationinvariance}
u^{\infty}_{S_h}(\hx,k)
&=&\int_{D_h}e^{-ik\hx\cdot\,y}S_h(y,k)dy\cr
&=&\int_{D}e^{-ik\hx\cdot\,(x+h)}S(x,k)dx\cr
&=&e^{-ik\hx\cdot\,h}u^{\infty}_{S}(\hx,k), \quad \hx\in \mathbb{S}^{n-1},\,k\in (k_{min},k_{max}).
\en
This implies that the modulus of the far field pattern is invariant under translations.

To find or describe the location of the unknown target, we add reference point sources into the scattering system and take the corresponding phaseless data.
Following the ideas given in our recent work \cite{JiLiuZhang} for phaseless inverse obstacle and medium problems, we then introduce a novel stable phase retrieval method using at most three point sources. We have the freedom to choose the locations of point sources, but use different scattering strengths.

The remaining part of the work is organized as follows. In the next section, we introduce the new phaseless multi-frequency inverse source problem with give point sources. A phase
retrieval method is the proposed. We show that our phase retrieval method is Lipschitz stability with respect to measurement noise. Section \ref{Uni} is devoted to some uniqueness
results on the new phaseless multi-frequency inverse source problem.
In Section \ref{DSMs}, we introduce a direct sampling method for source support reconstruction by using broadband sparse phaseless far field data directly. We also combine the phase retrieval method and the direct sampling method proposed in \cite{AlaHuLiuSun} to determine the source support. We want to strength that all the numerical methods make no explicit use of any a priori information of the source. These algorithms are then verified in Section \ref{NumExamples} by extensive examples in two dimensions.

\section{Phase retrieval method}
\label{Pre}\setcounter{equation}{0}

\subsection{New scattering model with given reference point sources}

Let $z_0\in\R^n\ba\ov{D}$ be a fixed point outside $D$. By adding a point source into the underlying scattering system, $u_{S, z_0}(x,k,\tau):= u_{S}(x,k) + \tau\Phi_k(x,z_0)$ is the unique solution to the problem
\ben
\Delta u_{S, z_0}(x,k,\tau) +k^{2} u_{S, z_0}(x,k,\tau)= S(x,k)-\tau\delta_{z_0} \quad \quad in \; \;  \mathbb{R}^{n}\\
\lim_{r\longrightarrow \infty} r^{\frac{n-1}{2}}\Big(\frac{\partial u_{S,z_0}}{\partial r} - iku_{S,z_0} \Big) = 0, \quad \quad r=|x|,
\enn
where $\tau\in \C$ is the scattering strength of the point source and $\delta_{z_0}=\delta(x-z_0)$ is the Dirac delta function at the point $z_0$. In particular, if $\tau=0$, such a problem is reduced to the scattering problem  \eqref{Helmholtzequation}-\eqref{SRC}.
The corresponding far field pattern $u^{\infty}_{S, z_0}$ takes the form
\be\label{FarFieldDplusz0}
u^{\infty}_{S,z_0}(\hx, k, \tau)=u^{\infty}_{S}(\hx, k)+\tau e^{-ik\hx\cdot z_0}, \quad \hx\in \mathbb{S}^{n-1},\,k\in(k_{min}, k_{max}),\,\tau\in\C.
\en

Straightforward calculations show that the phaseless data $|u_{S, z_0}|$ and $|u^{\infty}_{S,z_0}|$ are also invariant under rotations and the modulus of the far field data $|u^{\infty}_{S,z_0}|$ is invariant under translations. However, since we have the freedom to choose the source point $z_0$ and the scattering strength $\tau$, the inverse problem considered is modified as follows:\\

{\bf (IP2)}:{\quad\em Determine the source $S$ from one of the following data sets
\ben
\Big\{|u_{S,z_0}(x, k,\tau)|:\,x\in \G,\,k\in (k_{min},k_{max}),\,\tau\in\mathcal{T},\,z_0\in \mathcal {Z}\Big\};
\enn
or
\ben
\Big\{|u^\infty_{S,z_0}(\hx, k,\tau)|:\,\hx\in \mathbb{S}^{n-1},\,k\in (k_{min},k_{max}),\,\tau\in\mathcal{T},\,z_0\in \mathcal {Z}\Big\},
\enn
where $\mathcal {T}$ and $\mathcal {Z}$ are the admissible scattering strength set and source point set, respectively.
}\\

With properly chosen $\mathcal {T}$ and $\mathcal {Z}$, we introduce a novel phase retrieval method for the phased data in the next subsection. Unique determination of the source $S$
will be established in Section \ref{Uni}. Some numerical methods for the source support will be studied in Sections \ref{DSMs} and \ref{NumExamples}.

\subsection{Phase retrieval method}
In this subsection, we introduce a novel phase retrieval method to obtain the phased data $u_{S}$ and $u^{\infty}_S$ from the phaseless data $|u_{S,z_0}|$ and $|u^{\infty}_{S,z_0}|$, respectively.
Our method is based on the following geometric result \cite{JiLiuZhang}.

\begin{lemma}\label{phaseretrieval}
Let $z_j:=x_j+iy_j,\,j=1,2,3$ be three different complex numbers such that they are not collinear. Then the complex number $z\in\C$ is uniquely determined by the distances $r_j=|z-z_j|,\,j=1,2,3$.
\end{lemma}

Lemma \ref{phaseretrieval} ensures the uniqueness of the phase reconstruction. Numerically, we have the following phase retrieval scheme \cite{JiLiuZhang}.

\begin{figure}[htbp]
\centering
\includegraphics[width=3in]{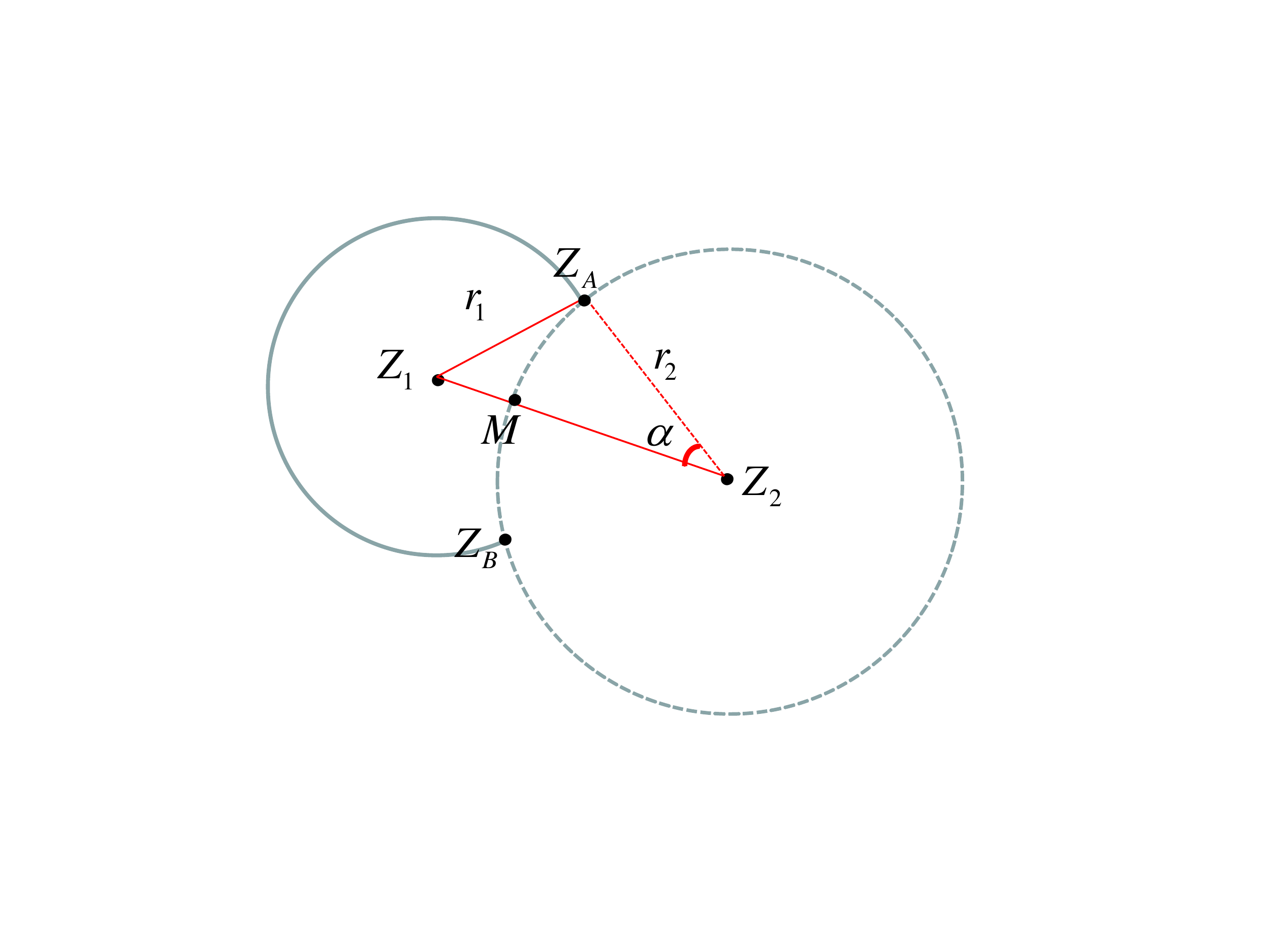}
\caption{Sketch map for phase retrieval scheme. Here, $Z_j=(x_j,y_j)$ is the point in the plane corresponding to the given complex number $z_j$, $j=1,2,3$. Define $Z=(x,y)$ to be the point corresponding to the unknown complex number $z$. Then $Z$ is located on the spheres $\pa B_{r_j}(Z_j)$ centered at $Z_j$ with radius $r_j,\,j=1,2,3$.}
\label{twodisks}
\end{figure}
{\bf Phase Retrieval Scheme. (Numerical simulation for Lemma \ref{phaseretrieval}.)}
\begin{itemize}{\em
  \item (1). {\bf\rm Collect the distances $r_j:=|z-z_j|$ with given complex numbers $z_j,\,j=1,2,3$.} If $r_j=0$ for some $j\in\{1,2,3\}$, then $Z=Z_j$. Otherwise, go to next step.
  \item (2). {\bf\rm Look for the point $M=(x_M, y_M)$.} As shown in Figure \ref{twodisks}, $M$ is the intersection of circle centered at $Z_2$ with radius $r_2$ and the ray $Z_2Z_1$ with initial point $Z_2$. Denote by $d_{1,2}:=|z_1-z_2|$ the distance between $Z_1$ and $Z_2$, then
       \be\label{xMyM}
       x_M=\frac{r_2}{d_{1,2}}x_1+\frac{d_{1,2}-r_2}{d_{1,2}}x_2,\quad y_M=\frac{r_2}{d_{1,2}}y_1+\frac{d_{1,2}-r_2}{d_{1,2}}y_2,
       \en
  \item (3). {\bf\rm Look for the points $Z_A=(x_A, y_A)$ and $Z_B=(x_B, y_B)$.} Note that $Z_A$ and $Z_B$ are just two rotations of $M$ around the point $Z_2$. Let $\alpha\in [0,\pi]$ be the angle between rays $Z_2Z_1$ and $Z_2Z_A$. Then, by the law of cosines, we have
       \be\label{cosalpha}
       \cos \alpha=\frac{r_1^2-r_2^2-d_{1,2}^2}{2r_2d_{1,2}}.
       \en
       Note that $\alpha\in [0,\pi]$ and $\sin^2 \alpha+\cos^2 \alpha=1$, we deduce that $\sin \alpha=\sqrt{1-\cos^2 \alpha}$.
       Then
       \be
       \label{xA}x_A &=& x_2+\Re\{[(x_M-x_2)+i(y_M-y_2)]e^{-i\alpha}\},\\
       \label{yA}y_A &=& y_2+\Im\{[(x_M-x_2)+i(y_M-y_2)]e^{-i\alpha}\},\\
       \label{xB}x_B &=& x_2+\Re\{[(x_M-x_2)+i(y_M-y_2)]e^{i\alpha}\},\\
       \label{yB}y_B &=& y_2+\Im\{[(x_M-x_2)+i(y_M-y_2)]e^{i\alpha}\}.
       \en
  \item (4). {\bf\rm Determine the point $Z$. $Z=Z_A$ if the distance $|Z_AZ_3|=r_3$, or else $Z=Z_B$.}}
\end{itemize}

Note that the measurements are always practically polluted by unavoidable errors. Thus we wish to approximate the phased data $z$ from a knowledge of the perturbed phaseless data
$r_j^{\eps}$ with a known error level
\ben
|r_j^{\eps}-r_j|\leq\eps, \quad j=1,2,3.
\enn
Here and throughout the paper, we use the right upper corner sign $\eps$ to denote the polluted data. From \eqref{xMyM}, we deduce that
\be\label{xMyMestimate}
|x_M^{\eps}-x_M|=\frac{|x_1-x_2|}{d_{1,2}}|r_2^{\eps}-r_2|\leq \eps\quad\mbox{and}\quad |y_M^{\eps}-y_M|=\frac{|y_1-y_2|}{d_{1,2}}|r_2^{\eps}-r_2|\leq \eps.
\en
Similarly, \eqref{cosalpha} implies the existence of a constant $c_1>0$ depending on $Z_j, j=1,2,3$, such that
\ben
|e^{i\alpha^{\eps}}-e^{i\alpha}|\leq c_1\eps.
\enn
Combining this with \eqref{xMyMestimate} and \eqref{xA}-\eqref{yB}, we find that
there exists a constant $c_2>0$ depending on $Z_j, j=1,2,3$, such that
\ben
|x_{ii}^{\eps}-x_{ii}|\leq c_2\eps\quad\mbox{and}\quad |y_{ii}^{\eps}-y_{ii}|\leq c_2\eps,\quad ii=A,B.
\enn
Therefore, we have
\be\label{phasestability}
|Z^{\eps}-Z|\leq \sqrt{2}c_2\eps.
\en
This implies that our phase retrieval scheme is Lipschitz stable with respect to the measurement noise level $\eps$.

For any fixed $z_0\in\R^n\ba\ov{D}$, let $\tau_j\in\C, j=1,2,3$ be three scattering strengths with different principal arguments. Choose
\ben
\mathcal {Z}:=\{z_0\},\,\mathcal {T}:=\{\tau_1, \tau_2, \tau_3\}.
\enn
For the case with phaseless scattered field data, we set
\ben
z_j:=-\tau_j\Phi_k(x,z_0)\quad\mbox{and}\quad r_j:=|u_{S,z_0}(x,k,\tau_j)|, \quad x\in\G, \,k\in (k_{min}, k_{max}),\, j=1,2,3.
\enn
For the case with phaseless far field data, we set
 \ben
z_j:=-\tau_j e^{-ik\hx\cdot z_0}\quad\mbox{and}\quad r_j:=|u^{\infty}_{S,z_0}(\hx,k,\tau_j)|, \quad \hx\in \mathbb{S}^{n-1}, \,k\in (k_{min}, k_{max}),\, j=1,2,3.
\enn
Denote by $Z_j=(x_j,y_j)$ the three points in the plane corresponding to the three given different complex numbers $z_j,\,j=1,2,3$.
By \eqref{phasestability}, we have the following stability result.

\begin{theorem}\label{phaseretrieval-stability}
For any fixed $z_0\in\R^n\ba\ov{D}$, let $\tau_j\in\C, j=1,2,3$ be three scattering strengths with different principal arguments. Assume that we have the measured phaseless data $|u^{\eps}_{S, z_0}|$ or $|u^{\infty,\eps}_{S,z_0}|$ with
\ben
\Big||u^{\eps}_{S, z_0}(x,k,\tau_j)|-|u_{S, z_0}(x,k,\tau_j)|\Big|\leq\eps, \,x\in\G,\,k\in (k_{min}, k_{max}),\,j=1,2,3,
\enn
or
\ben
\Big||u^{\infty,\eps}_{S, z_0}(\hx,k,\tau_j)|-|u^{\infty}_{S, z_0}(\hx,k,\tau_j)|\Big|\leq\eps, \,\hx\in \mathbb{S}^{n-1},\,k\in (k_{min}, k_{max}),\,j=1,2,3.
\enn
Then we have
\ben
\Big|u^{\eps}_{S}(x,k)-u_{S}(x,k)|\Big|\leq C\eps,\,x\in\G,\,k\in (k_{min}, k_{max}),
\enn
or
\ben
\left|u^{\infty,\eps}_{S}(\hx,k)-u^{\infty}_{S}(\hx,k)\right|\leq C\eps, \,\hx\in \mathbb{S}^{n-1},\,k\in (k_{min}, k_{max})
\enn
for some constant $C>0$ depending only on $\tau_j, j=1,2,3$.
\end{theorem}

\section{Uniqueness}
\label{Uni}

In this section, we study the uniqueness for the inverse problem {\bf (IP2)}, i.e., we investigate the question whether the phaseless data provides enough information to determine the nature of the scatterer $S$ completely.  We assume that the source $S(y,k)$ is a product of a spatial function $f$ and a frequency function $g$, i.e.,
\be\label{Sfg}
S(y,k)=f(y)g(k), \quad y\in \R^n, k\in (k_{min},k_{max}).
\en
Here, $g\in C(k_{min},k_{max})$ is a given nontrivial function of $k$. Thus, there exists an interval $(k_a,k_b)\subset(k_{min},k_{max})$
such that
\be\label{gassumption}
g(k)\neq 0, \quad k\in (k_a,k_b).
\en
The following theorem establish the uniqueness of the source with phased data. We refer to \cite{BaoLiLinTriki,WangGuoZhangLiu2017IP} for the special case with $g(k)\equiv 1$, i.e., the source $S$ is independent of the wave number $k$.

\begin{theorem}\label{uni-full}
Then the source $S$ is uniquely determined by one of the following phased data set:
\ben
(1).\, \Big\{u_{S}(x, k):\,x\in \G,\,k\in (k_{min},k_{max})\Big\}, \quad (2).\, \Big\{u^{\infty}_{S}(\hx, k):\,\hx\in \mathbb{S}^{n-1},\,k\in (k_{min},k_{max})\Big\}.
\enn
\end{theorem}
Note that by Rellich's lemma and unique continuation principle, the above mentioned two data sets are actually equivalent. In the sequel, we consider only the phaseless far field data for the inverse problem {\bf (IP2)}. The uniqueness results with phaseless scattered data can be proven quite similarly.
\begin{proof}
Recall the representation \eqref{uinf} of the far field pattern, we have
\ben
u^{\infty}_{S}(\hx,k)=\int_{\R^n}e^{-ik\hx\cdot\,y}f(y)g(k)dy,\quad \hx\in\,\mathbb{S}^{n-1}, k\in (k_{min},k_{max}).
\enn
Note that $g(k)\neq 0$ for $k\in (k_a,k_b)$, define $\xi:=k\hx$, we obtain the data
\ben
\wi{f}(\xi):=\int_{\R^n}e^{-i\xi\cdot\,y}f(y)dy, \quad \xi\in B_{k_b}\ba\ov{B_{k_a}}.
\enn
Here, $B_r$ is the ball with radius $r$ centered at the origin. Clearly, $\wi{f}$ is an analytic function on $\xi$, thus we actually obtain the following data
\ben
\wi{f}(\xi):=\int_{\R^n}e^{-i\xi\cdot\,y}f(y)dy, \quad \xi\in \R^n,
\enn
which is the Fourier transform of the function $f$. Thus, by the Fourier integral theorem, the spacial function $f$ is uniquely determined.
\end{proof}

\begin{theorem}\label{uni-phaseless012}
Let $\tau:=|\tau|e^{i\alpha}\in\C\ba\{0\}$ with the principal argument $\alpha\in [0,2\pi)$, and let $z_0$ and $z_1$ be two different points in $\R^n\ba\ov{D}$.
Then the source $S$ is uniquely determined by the phaseless far field data set
$\big\{|u^{\infty}_{S, p}(\hx, k, \tau)|:\,\hx \in \mathbb{S}^{n-1},\,k\in (k_{min}, k_{max}),\, \tau\in \mathcal \{0, \tau_1\}, p\in \{z_0,z_1\}\big\}$.
\end{theorem}
\begin{proof}
For all $\hx \in \mathbb{S}^{n-1},\,k\in (k_{min}, k_{max})$ and $p\in \{z_0,z_1\}$, we have
\ben
|u^{\infty}_{S, p}(\hx, k, \tau)|^2
&=&|u^{\infty}_{S}(\hx, k)+\tau e^{-ik\hx\cdot p}|^2\cr
&=&|u^{\infty}_{S}(\hx, k)|^2+2\Re\big(u^{\infty}_{S}(\hx, \hth)\ov{\tau e^{-ik\hx\cdot p}}\big)+|\tau|^2.
\enn
This implies $\Re\big(u^{\infty}_{S}(\hx, \hth)\ov{\tau e^{-ik\hx\cdot p}}\big)$ can be uniquely determined for all $\hx \in \mathbb{S}^{n-1},\,k\in (k_{min}, k_{max})$ and all $p\in \{z_0,z_1\}$. We rewrite $u^{\infty}_{S}(\hx, k)$ and $\tau e^{-ik\hx\cdot p}$ in the form
\ben
u^{\infty}_{S}(\hx, k)=|u^{\infty}_{S}(\hx, k)|e^{i\phi(\hx,k)} \quad\mbox{and}\quad \tau e^{-ik\hx\cdot p}=|\tau|e^{i(\alpha-k\hx\cdot p)},
\enn
respectively, for all $\hx \in \mathbb{S}^{n-1},\,k\in (k_{min}, k_{max})$ and all $p\in \{z_0,z_1\}$.
Here, $\phi(\hx,k)\in [0,2\pi)$ is the principal argument of $u^{\infty}_{S}(\hx, k)$.
We claim that the argument $\phi(\hx,k)$ of $u^{\infty}_{S}(\hx, k)$ is uniquely determined.
Indeed, assume that there are two arguments $\phi_1(\hx,k)$ and $\phi_2(\hx,k)$.
Define
\ben
\mathbb{S}^{n-1}_0(k,z_0,z_1):=\{\hx\in \mathbb{S}^{n-1}:\, u_{S}^{\infty}(\hx,k)\neq 0\quad\,\mbox{and}\quad\, \hx\cdot(z_0-z_1)\neq 0\}.
\enn
Since $|u^{\infty}_{S}(\hx, k)|$ is given in the data set and $\Re\big(u^{\infty}_{S}(\hx, \hth)\ov{\tau e^{-ik\hx\cdot p}}\big)$ can be obtained for all $\hx \in \mathbb{S}^{n-1},\,k\in (k_{min}, k_{max})$ and all $p\in \{z_0,z_1\}$,
we have that for all $\hx \in \mathbb{S}^{n-1}_0(k,z_0,z_1),\,k\in (k_{min}, k_{max})$ and all $p\in \{z_0,z_1\}$,
\ben
\cos[\phi_1(\hx,k)-(\alpha-k\hx\cdot p)]=\cos[\phi_2(\hx,k)-(\alpha-k\hx\cdot p)],
\enn
and furthermore, we conclude that
\be\label{case1}
\phi_1(\hx,k)-(\alpha-k\hx\cdot p)=\phi_2(\hx,k)-(\alpha-k\hx\cdot p)+2m\pi,\quad \mbox{for\,\,some}\, m\in\Z,
\en
or
\be\label{case2}
\phi_1(\hx,k)-(\alpha-k\hx\cdot p)=-[\phi_2(\hx,k)-(\alpha-k\hx\cdot p)]+2l\pi,\quad\mbox{for\,\,some}\, l\in\Z.
\en

We now show that the case \eqref{case2} does not hold. Actually, \eqref{case2} implies that
\be\label{phiplus}
\phi_1(\hx,k)+\phi_2(\hx,k)-2l\pi=2[\alpha-k\hx\cdot p],\quad p\in \{z_0, z_1\}, \,\mbox{for\,\,some}\,l\in\Z.
\en
The left side of \eqref{phiplus} is independent of $p$. However, the right side of \eqref{phiplus} changes for different $p\in \{z_0, z_1\}$. This leads to a contradiction, and thus \eqref{case2} does not hold.

For the case when \eqref{case1} holds, we have
\ben
\phi_1(\hx,k)-\phi_2(\hx,k)=2m\pi,\quad \mbox{for\,\,some}\,m\in\Z.
\enn
Noting that $\phi_1(\hx,k),\phi_2(\hx,k)\in [0,2\pi)$, we have $\phi_1(\hx,k)-\phi_2(\hx,k)\in (-2\pi,2\pi)$, and thus $m=0$, i.e.,
\ben
\phi_1(\hx,k)=\phi_2(\hx,k), \quad\forall\,\hx \in \mathbb{S}^{n-1}_0(k,z_0,z_1),\, k\in (k_{min}, k_{max}).
\enn
This further implies that $u^{\infty}_{S}(\hx, k)$ is uniquely determined for all $\hx \in \mathbb{S}^{n-1}_0(k,z_0,z_1),\,k\in (k_{min}, k_{max})$ and also for
$\hx \in \mathbb{S}^{n-1},\,k\in (k_{min}, k_{max})$ by analytic continuation. The proof is now completed by applying Theorem \ref{uni-full}.
\end{proof}

Theorem \ref{uni-phaseless01} below shows that, under further assumption on the source, uniqueness can even be established by using much less data.

\begin{theorem}\label{uni-phaseless01}
Let the source $S$ take the form \eqref{Sfg} with a spatial function $f$ and a given real valued continuous function $g$ satisfying \eqref{gassumption}.
In $\R^3$, we further assume that
\be\label{Rs>0}
\Re(f(y))\geq0 \quad\mbox{and}\quad \exists\, y_0\in\R^3\,\, s.t.\,\, \Re(f(y_0))>0.
\en
In $\R^2$, we further assume that
\be\label{is>0}
\Im(f(y)\geq0 \quad\mbox{and}\quad \exists\, y_0\in\R^2\,\, s.t.\,\, \Im(f(y_0))>0.
\en
Then, for any fixed $\tau_1\in\R\ba\{0\}$, the source $S$ is uniquely determined by the phaseless data set
$\big\{|u^{\infty}_{S, z_0}(\hx, k, \tau)|:\,\hx \in \mathbb{S}^{n-1},\,k\in (k_{min}, k_{max}),\, \tau\in \{0, \tau_1\}\big\}$.
\end{theorem}
\begin{proof}
From the same arguments as in Theorem \ref{uni-phaseless012}, we deduce that $\Re\big(u^{\infty}_{S}(\hx, k)\ov{\tau_{1}e^{-ik\hx\cdot z_0}}\big)$ is uniquely determined for all $\hx \in \mathbb{S}^{n-1},\,k\in (k_{min}, k_{max})$.
For the fixed point $z_0\in \R^{n}\ba\ov{D}$, define $S_{-z_0}(y,k):=S(y+z_0,k)$. Then by the translation relation \eqref{translationinvariance}, we have
\be\label{uinf-z0}
u^{\infty}_{S_{-z_0}}(\hx,k)=e^{ikz_{0}\cdot\hx}u^{\infty}_{S}(\hx,k)\quad \forall\hx\in \mathbb{S}^{n-1},\,k\in (k_{min}, k_{max}).
\en
This implies that
\ben
\Re\big(u^{\infty}_{S_{-z_0}}(\hx, k)\ov{\tau_1}\big)=\Re\big(u^{\infty}_{S}(\hx, k)\ov{\tau_{1}e^{-ik\hx\cdot z_0}}\big),
\enn
and thus $\Re\big(u^{\infty}_{S_{-z_0}}(\hx, k)\ov{\tau_1}\big)$ is uniquely determined for all $\hx\in \mathbb{S}^{n-1},\,k\in (k_{min}, k_{max})$.
From \eqref{uinf-z0}, we have $|u^{\infty}_{S_{-z_0}}(\hx, k)|=|u^{\infty}_{S}(\hx, k)|$. We rewrite $u^{\infty}_{S_{-z_0}}(\hx, k)$ in the form
\be\label{uinfphi}
u^{\infty}_{S_{-z_0}}(\hx, k)=|u^{\infty}_{S}(\hx, k)|e^{i\phi(\hx,k)},
\en
with some analytic function $\phi$ such that $\phi(\hx,k)\in [0,2\pi)$ for all $\hx \in \mathbb{S}^{n-1},\,k\in (k_{min}, k_{max})$. We claim that $\phi$ is uniquely determined.
If so, $u^{\infty}_{S_{-z_0}}$ is uniquely determined, and thus the source $S_{-z_0}$ is uniquely determined by using Theorem \ref{uni-full}. Furthermore, the source $S$ is uniquely determined since $S$ is just the translation of $S_{-z_0}$. The proof is complete by showing that $\phi$ is uniquely determined.

Assume that there are two functions $\phi_1$ and $\phi_2$.
To simplify the notations, for any $k\in (k_{min}, k_{max})$, we define
\ben
\mathbb{S}^{n-1}_0(k):=\{\hx\in \mathbb{S}^{n-1}: u^{\infty}_{S}(\hx, k)\neq 0\}.
\enn
Then by analyticity of $u^{\infty}_{S}(\hx, k)$, we deduce that the set $\mathbb{S}^{n-1}_0(k)$ has Lebesgue measure zero.
Since $|u^{\infty}_{S}(\hx, k)|$ is given and $\Re\big(u^{\infty}_{S_{-z_0}}(\hx, k)\ov{\tau}\big)$ is obtained for all $\hx\in \mathbb{S}^{n-1}, k\in (k_{min}, k_{max})$, we have
\be\label{cos1cos2}
\cos[\phi_1(\hx,k)]=\cos[\phi_2(\hx,k)], \quad \hx\in \mathbb{S}^{n-1}_0(k), \,k\in (k_{min}, k_{max}).
\en
This implies that for any fixed $\hx\in \mathbb{S}^{n-1}_0(k), \,k\in (k_{min}, k_{max})$, $\phi_1(\hx,k)=\phi_2(\hx,k)$ or $\phi_1(\hx,k)=2\pi-\phi_2(\hx,k)$  since $\phi_1(\hx,\hth), \phi_2(\hx,\hth)\in [0,2\pi)$.

We claim that
\be\label{phi1phi2}
\phi_1(\hx,k)=\phi_2(\hx,k), \quad \hx\in \mathbb{S}^{n-1}_0(k), \,k\in (k_{min}, k_{max}),
\en
or
\be\label{phi1-phi2}
\phi_1(\hx,k)=2\pi-\phi_2(\hx,k), \quad \hx\in \mathbb{S}^{n-1}_0(k), \,k\in (k_{min}, k_{max}).
\en
For the special case $\phi_1(\hx,k)\equiv\pi$ for all $\hx\in \mathbb{S}^{n-1}_0(k), \,k\in (k_{min}, k_{max})$, we have $\phi_2\equiv\phi_1=\pi$ by \eqref{cos1cos2}. Otherwise, without loss of generality, there exists a direction $\hx_1\in \mathbb{S}^{n-1}_0(k_1)$ for some wave number $k_1\in (k_{min}, k_{max})$ such that
$\phi_1(\hx_1,k_1)>\pi$. By analyticity of the function $\phi_1$, there exists a neighbourhood $\mathbb{U}(\hx_1)\subset \mathbb{S}^{n-1}_0(k_1)$ of $\hx_1$ such that
\be\label{phi1>pi}
\phi_1(\hx,k_1)>\pi,\quad\forall \hx\in \mathbb{U}(\hx_1).
\en
If $\phi_2(\hx_1,k_1)=\phi_1(\hx_1,k_1)>\pi$, then $\phi_2(\hx,k_1)=\phi_1(\hx,k_1),\quad\forall \hx\in \mathbb{U}(\hx_1)$.
Otherwise, there exists a  direction $\hx_2\in \mathbb{U}(\hx_1)$ such that $\phi_2(\hx_2,k_1)=2\pi-\phi_1(\hx_2,k_1)$.
From \eqref{phi1>pi}, we have $\phi_2(\hx_2,k_1)=2\pi-\phi_1(\hx_2,k_1)<\pi$. By analyticity of $\phi_2$ in $\mathbb{U}(\hx_1)$,
there exists a direction $\hx_0\in \mathbb{U}(\hx_1)$ such that $\phi_2(\hx_0,k_1)=\pi$. Thus $\phi_1(\hx_0,k_1)=\phi_2(\hx_0,k_1)=\pi$
by \eqref{cos1cos2}. However, this leads to a contradiction to \eqref{phi1>pi}.

We show that \eqref{phi1-phi2} does not hold. Let $S_1:=f_1g$ and $S_2:=f_2g$ be sources corresponding to $\phi_1$ and $\phi_2$, respectively. With the help of unique continuation, from \eqref{uinfphi} and \eqref{phi1-phi2} we have
\ben
u^{\infty}_{S_1}(\hx, k)= \ov{u^{\infty}_{S_2}(\hx, k)}, \quad \hx\in \mathbb{S}^{n-1}, \,k\in (k_{min}, k_{max}).
\enn
Using the representation \eqref{uinf} and the assumption that the function $g$ is a given real valued function satisfying \eqref{gassumption}, we have
\be\label{f1f2}
\int_{\R^n}e^{-ik\hx\cdot y} f_1(y)dy
&=& \ov{\int_{\R^n}e^{-ik\hx\cdot y} f_2(y)dy}\cr
&=& \int_{\R^n}e^{ik\hx\cdot y} \ov{f_2(y)}dy \cr
&=& (-1)^n\int_{\R^n}e^{-ik\hx\cdot y} \ov{f_2(-y)}dy, \quad \hx\in \mathbb{S}^{n-1}, \,k\in (k_{a}, k_{b}).
\en
Clearly, both sides are analytic on the wave number $k$ and thus \eqref{f1f2} holds for all $\hx\in \mathbb{S}^{n-1},\, k\in [0,\infty)$.
By the Fourier integral theorem, we conclude that
\be\label{S1S2}
f_1(y)=(-1)^n \ov{f_2(-y)}, \quad y\in \R^n.
\en
In $\R^3$, from the assumption \eqref{Rs>0}, we have
\be\label{S1y0}
\Re(f_1(y_0))>0 \quad\mbox{for\,some}\,\,y_0\in \R^3.
\en
However, from \eqref{S1S2}, using the assumption \eqref{Rs>0} on $S_2$, we deduce that
\ben
\Re(f_1(y)) = (-1)\Re(f_2(-y))\leq 0, \forall y\in \R^3.
\enn
This is a contradiction to \eqref{S1y0}. In $\R^2$, we can derive similar contradiction. Thus \eqref{phi1-phi2} does not hold, and we deduce that \eqref{phi1phi2}, i.e.,
$\phi_1=\phi_2$. The proof is complete.
\end{proof}

Assumptions \eqref{Rs>0} and \eqref{is>0} are used to show that \eqref{phi1-phi2} does not hold. Theorem \ref{uni-phaseless012} also holds if the assumptions \eqref{Rs>0} and \eqref{is>0} are replaced by
\be\label{Rs<0}
\Re(f(y))\leq0 \quad\mbox{and}\quad \exists\, y_0\in\R^3\, s.t. \Re(f(y_0))<0,
\en
and
\be\label{is<0}
\Im(f(y)\leq0 \quad\mbox{and}\quad \exists\, y_0\in\R^2\, s.t. \Im(f(y_0))<0,
\en
respectively. We think these assumptions can be removed. But this need different techniques.

In Theorems \ref{uni-phaseless012} and \ref{uni-phaseless01}, the phaseless data set includes the data with scattering strength $\tau=0$.
For diversity, we now prove a uniqueness theorem with three different scattering strength.

\begin{theorem}\label{uni-phaseless123}
Let $\tau_1, \tau_2, \tau_3\in\C$ be such that $\tau_2-\tau_1$ and $\tau_3-\tau_1$ are linearly independent.
Then the source $S$ is uniquely determined by the phaseless data set
$\big\{|u^{\infty}_{S,z_0}(\hx, k, \tau_j)|:\,\hx\in\mathbb{S}^{n-1},\,k\in (k_{min}, k_{max}), j=1,2,3\big\}$, where $z_0\in\R^n\ba\ov{D}$ is a fixed point outside $D$.
\end{theorem}
\begin{proof}
For any fixed $\hx\in\mathbb{S}^{n-1},\,k\in (k_{min}, k_{max})$, define
\ben
z_j(\hx,k):=\tau_{j}e^{-ik\hx\cdot z_0},\,j=1,2,3.
\enn
Obviously, $z_j$ is just the counterclockwise rotation of $\tau_j$ through the angle $\hth_0:=k(\hth-\hx)\cdot z_0$ about the origin for all $j=1,2,3.$
Thus, by the assumption on $\tau_j$, we have that the complex numbers $z_j,\,j=1,2,3$ are not collinear. Using Lemma \ref{phaseretrieval} and the data
\ben
|u^{\infty}_{S, z_0}(\hx, k, \tau_j)| = |u^{\infty}_{S}(\hx, k)+z_j(\hx,k)|,\quad j=1,2,3,
\enn
we found that the far field patterns $u^{\infty}_{S}(\hx, k)$ for all $\hx\in\mathbb{S}^{n-1},\,k\in (k_{min}, k_{max})$ can be uniquely determined. Then the proof is finished by Theorem \ref{uni-full}.
\end{proof}

We want to remark that by analyticity the far field pattern can be determined on the whole unit sphere $\mathbb{S}^{n-1}$ using partial value on some surface patch of $\mathbb{S}^{n-1}$.
In the proofs of Theorems \ref{uni-phaseless012}, \ref{uni-phaseless01} and \ref{uni-phaseless123},  we have showed that the far field pattern $u^{\infty}_S$ can be determined uniquely. Thus the phaseless data can be taken on any nonempty surface patch of $\mathbb{S}^{n-1}$.

We are also interested in broadband sparse measurements, i.e., the data can only be measured in finitely many observation directions,
\ben
\{\hx_1, \hx_2, \cdots, \hx_M\}=:\Theta\subset \mathbb{S}^{n-1}.
\enn
In particular, we are interested in what information can be obtained using multi-frequency data with a single observation direction.
For a bounded domain $D$, the $\hx$-strip hull of $D$ for a single observation direction $\hx \in \,\Theta$ is defined by
\ben
S_{D}(\hx):=  \{ y\in \mathbb \R^{n}\; | \; \inf_{z\in D}z\cdot \hx \leq y\cdot \hx \leq \sup_{z\in D}z\cdot \hx\},
\enn
which is the smallest strip (region between two parallel hyper-planes) with normals in the directions $\pm \hx$ that contains $\overline{D}$.
Let
\ben
\Pi_{\alpha}:=\{y\in \R^n | y\cdot\hx+\alpha=0\}
\enn
be a hyperplane with normal $\hx$. Define
\be
\hat{f}(\alpha):=\int_{\Pi_\alpha}f(y)ds(y).
\en
In a recent work \cite{AlaHuLiuSun}, under the assumption that the set
\be\label{set2}
\{\alpha\in\R|\, \Pi_\alpha\subset S_D(\hx), \hat{f}(\alpha)=0\}
\en
has Lebesgue measure zero, it is shown that the strip $S_D(\hx)$ can be uniquely determined by the phased far field data set
$\{u^{\infty}_{S}(\hx, k): \, k\in (k_a, k_b)\}$. Based on this fact, following exactly the same line of arguments that are used in the proof of Theorem \ref{uni-phaseless123}, using
Lemma \ref{phaseretrieval} again, we have the following uniqueness result with phaseless far field data.

\begin{theorem}\label{uni-strip}
Let $\tau_1, \tau_2, \tau_3\in\C$ such that $\tau_2-\tau_1$ and $\tau_3-\tau_1$ are linearly independent.
If the set \eqref{set2} has Lebesgue measure zero, then, for any fixed $\hx\in\mathbb{S}^{n-1}$, the strip $S_D(\hx)$ of the source support $D$ can be uniquely determined by the phaseless far field data set $\big\{|u^{\infty}_{S,z_0}(\hx, k, \tau_j)|:\,k\in (k_{min}, k_{max}), j=1,2,3\big\}$.
\end{theorem}

A by-product of Theorem \ref{uni-strip} is that one may determine a convex hull of the source support $D$ by using two or more finite observation directions. This will be verified numerically in the next sections. Similarly, for the case with broadband sparse phaseless scattered data, we can show that, under certain conditions, the smallest annular centered at the measurement point containing the source support can be uniquely determined. We refer to \cite{AlaHuLiuSun} for more details on uniqueness result by using the phased data.

\section{Direct sampling methods with broadband sparse data}
\label{DSMs}
The uniqueness results in the previous section ensure the possibility to reconstruct the unknown objects by stable algorithms. In this section, we investigate the numerical methods for support reconstruction of the source $S$ by using phaseless far field data $\big| u^{\infty}_{S, z_0} \big|$. We will focus on designing the direct sampling methods
which do not need any a priori information on the geometry and physical properties of the obstacle.

Denote by $\Theta$ a finite set with finitely many observation directions as elements. We first introduce an auxiliary function
\be\label{G}
G(z,\hx):=\int_{k_{min}}^{k_{max}}u^{\infty}_{S}(\hx,k)e^{ik\hx\cdot z}dk, \quad\,z\in\R^n,\,\hx\in \Theta.
\en
Let $\hx^{\perp}$ be a direction perpendicular to $\hx$ and we have
\be\label{I1behavior1}
G(z+\alpha\hx^{\perp}, \hx)=G(z,\hx), \quad \forall z\in\R^n, \,\alpha\in\R,
\en
since $\hx^{\perp}\cdot \hx=0$.
This further implies that the functional $G$ has the same value for sampling points in the hyperplane with normal direction $\hx$.
By the well known Riemann-Lebesgue Lemma, we obtain that $G$ tends to $0$ as $|z|\rightarrow\infty$.

In the sequel, we assume that the source $S$ takes the form
\ben
S(y,k)=f(y)g(k), \quad y\in\R^n, k\in (k_{min},k_{max})
\enn
with $f\in L^{\infty}(\R^n),\,g\in C^1(0,\infty)$.
Recall from \eqref{uinf} that the far field pattern has the following representation
\ben
 u^{\infty}_{S}(\hx,k)=g(k)\int_{D}e^{-ik\hx\cdot\,y}f(y)dy,\quad\hx\in\Theta, \,k\in (k_{min},k_{max}).
\enn
Inserting it into the indicator $G$ defined in \eqref{G}, changing the order of integration, and integrating by parts, we have
\be\label{G2}
G(z,\hx)
&=&\int_{D}\int_{k_{min}}^{k_{max}}e^{ik\hx\cdot\,(z-y)} f(y)g(k)dk dy \cr
&=&\int_{D}\frac{\mathcal {S}_z(y,\hx)}{i\hx\cdot\,(z-y)}dy,
\en
where $S_z\in L^{\infty}(D)$ is given by
\ben
{S}_z(y,\hx):=f(y)\left[g(k)e^{ik\hx\cdot(z-y)}\Big|^{k_{max}}_{k_{min}}-\int_{k_{min}}^{k_{max}}e^{ik\hx\cdot(z-y)}g'(k)dk\right].
\enn
This implies that the functional $G$ is a superposition of functions that decays as $1/|\hx\cdot(z-y)|$ as the sampling point $z$ goes away from the strip $S_{D}(\hx)$.

For any $\hx\in \mathbb{S}^{n-1},\,k\in (k_{min}, k_{max}),\,\tau\in\C$, we define
\be\label{F}
\mathcal {F}(\hx,k,z_0,\tau)
&:=&|u^{\infty}_{S,z_0}(\hx, k, \tau)|^2-|u^{\infty}_{S}(\hx, \hth)|^2-|\tau|^2\cr
&=&|u^{\infty}_{S}(\hx, k)+\tau e^{-ikz_0\cdot\hx}|^2-|u^{\infty}_{S}(\hx, k)|^2-|\tau|^2\cr
&=&u^{\infty}_{S}(\hx, k)\ov{\tau}e^{ikz_0\cdot\hx}+\ov{u^{\infty}_{S}(\hx, k)}\tau e^{-ikz_0\cdot\hx}.
\en
Then, for any fixed $\tau\in \C\ba\{0\}$ and $z_0\in\R^n\ba\ov{D}$, we introduce the following indicator
\be
\label{Indicator01}{\bf I^{\Theta}_{z_0}}(z)&:=&\sum_{\hx\in\Theta}\left|\int_{k_{min}}^{k_{max}} \mathcal {F}(\hx,k,z_0,\tau)\cos[k\hx\cdot(z-z_0)]dk\right|,\quad\,z\in\R^n.
\en
Inserting \eqref{F} into \eqref{Indicator01}, straightforward calculations show that
\ben
{\bf I^{\Theta}_{z_0}}(z)=\sum_{\hx\in\Theta}\left| V_{z_0}(z,\hx)+\ov{V_{z_0}(z,\hx)}\right|
\enn
with
\ben
V_{z_0}(z,\hx):=\frac{\ov{\tau}}{2}\Big[G(z,\hx)+G(2z_0-z,\hx)\Big].
\enn
Let $D(z_0)$ be the point symmetric domain of $D$ with respect to $z_0$.
We expect that the indicator ${\bf I^{\Theta}_{z_0}}$ takes its local maximum on the locations of $D$ and $D(z_0)$.

Note that the indicator ${\bf I^{\Theta}_{z_0}}$ produces a false scatterer $D(z_0)$.
However, since we have the freedom to choose the point $z_0$, we can always choose $z_0$ such that the false domain $D(z_0)$ located outside our interested searching domain.
One may also overcome this problem by considering another indicator  ${\bf I^{\Theta}_{z_1}}$ with $z_1\in \R^n\ba\ov{D}$ and $z_1\neq z_0$.
So the scatterer $D$ can be determined numerically by the phaseless data set
$\big\{|u^{\infty}_{D\cup\{p\}}(\hx, k, \tau)|:\,\hx\in\mathbb{S}^{n-1}, k\in (k_{min}, k_{max}),\,\tau\in\{0,\tau_1\}, p\in\{z_0,z_1\}\big\}$. This is just the second phaseless data set
provided in Theorem \ref{uni-phaseless012}.
\quad\,\\

{\bf Scatterer Reconstruction Scheme One.}
\begin{itemize}{\em
  \item (1). Collect the phaseless data set $\big\{|u^{\infty}_{S, z_0}(\hx, k, \tau)|:\,\hx\in\Theta, \, k\in (k_{min}, k_{max}),\, \tau\in\{0,\tau_1\}\big\}$.
  \item (2). Select a sampling region in $\R^{n}$ with a fine mesh $\mathcal {T}$ containing the scatterer $D$,
  \item (3). Compute the indicator functional ${\bf I^{\Theta}_{z_0}}$  for all sampling point $z\in\mathcal {T}$,
  \item (4). Plot the indicator functional ${\bf I^{\Theta}_{z_0}}$ }.
\end{itemize}\,\quad\\

Using the {\bf Phase Retrieval Scheme} proposed in the previous Section, we obtain the phased far field pattern $u_{S}^{\infty}$. Then we have the second scatterer reconstruction algorithm using the following indicator \cite{AlaHuLiuSun}
\be\label{indicatorM}
I_2(z):=\sum_{\hx\in\Theta}\Big|G(z,\hx)\Big|,\quad z\in\R^n.
\en

{\bf Scatterer Reconstruction Scheme Two.}
\begin{itemize}{\em
  \item (1). Collect the phaseless data set
        $\big\{|u^{\infty}_{S, z_0}(\hx, k, \tau)|:\,\hx\in\Theta, \, k\in (k_{min}, k_{max}),\,  \tau\in\{\tau_1,\tau_2,\tau_3\}\big\}$.
  \item (2). Use the {\bf Phase Retrieval Scheme} to obtain the phased far field patterns $u^{\infty}_{S}(\hx, k)$ for all $\hx\in\Theta, \, k\in (k_{min}, k_{max})$,
  \item (3). Select a sampling region in $\R^{n}$ with a fine mesh $\mathcal {T}$ containing $D$,
  \item (4). Compute the indicator functional ${\bf I_2}(z)$ for all sampling point $z\in\mathcal {T}$,
  \item (5). Plot the indicator functional ${\bf I_2}(z)$.}
\end{itemize}


\section{Numerical examples and discussions}
\label{NumExamples}
\setcounter{equation}{0}

Now we present a variety of numerical examples in two dimensions to illustrate the applicability and effectiveness of our sampling methods with broadband sparse data.
The forward problems are computed the same as in \cite{AlaHuLiuSun}. For all examples, for $\hx \in \Theta$, we assume to have multiple frequency phaseless far field
data $u_S^\infty(\hx,k_j),\quad j=1,\cdots,N,$
where $N=20,k_{min}=0.5,k_{max}=20$ such that  $k_j=(j-0.5)\Delta k, \Delta k=\frac{k_{max}}{N}$.
We further perturb this data by random noise
\ben
\Big|u_{S,z_0}^{\infty,\delta}(\hx, k_l,\tau)\Big| = |u_{S,z_0}^{\infty}(\hx, k_l,\tau)| (1+\delta*e_{rel}), \quad l=1,2,\cdots, N,
\enn
where $e_{rel}$ is a uniformly distributed random number in the open interval $(-1,1)$. The value of $\delta$ used in our code is the relative error level.
We also consider absolute error in {\bf Example PhaseRetrieval}. In this case, we perturb the phaseless data
\ben
\Big|u_{S,z_0}^{\infty,\delta}(\hx, k_l,\tau)\Big| = \max\Big\{0,\,|u_{S,z_0}^{\infty}(\hx, k_l,\tau)|+\delta*e_{abs}\Big\},\quad l=1,2,\cdots, N,
\enn
where $e_{abs}$ is again a uniformly distributed random number in the open interval $(-1,1)$. Here, the value $\delta$ denotes the total error level in the measured data.
.

\subsection{${\bf I^{\Theta}_{z_0}}$ with one and two observation directions} We first consider the case of one observation using $\tau=1$ and different $z_0$. Let $S=5$ and the support of
$S$ is a rectangle given by $(1, 2) \times(1,1.6)$. In Fig. \ref{rec1}, we plot the indicators using $\hx= (1,0)$ and three reference points $z_0 = (1.5,4), z_0 = (4,4)$ and $z_0 = (12,12)$. The picture clearly shows that the source support and the corresponding point symmetric domain (with respect to $z_0$) lies in a strip, which is perpendicular to the observation direction.
\begin{figure}[htbp]
  \centering
  \subfigure[\textbf{$z_0=(1.5,4)$.}]{
    \includegraphics[width=2in]{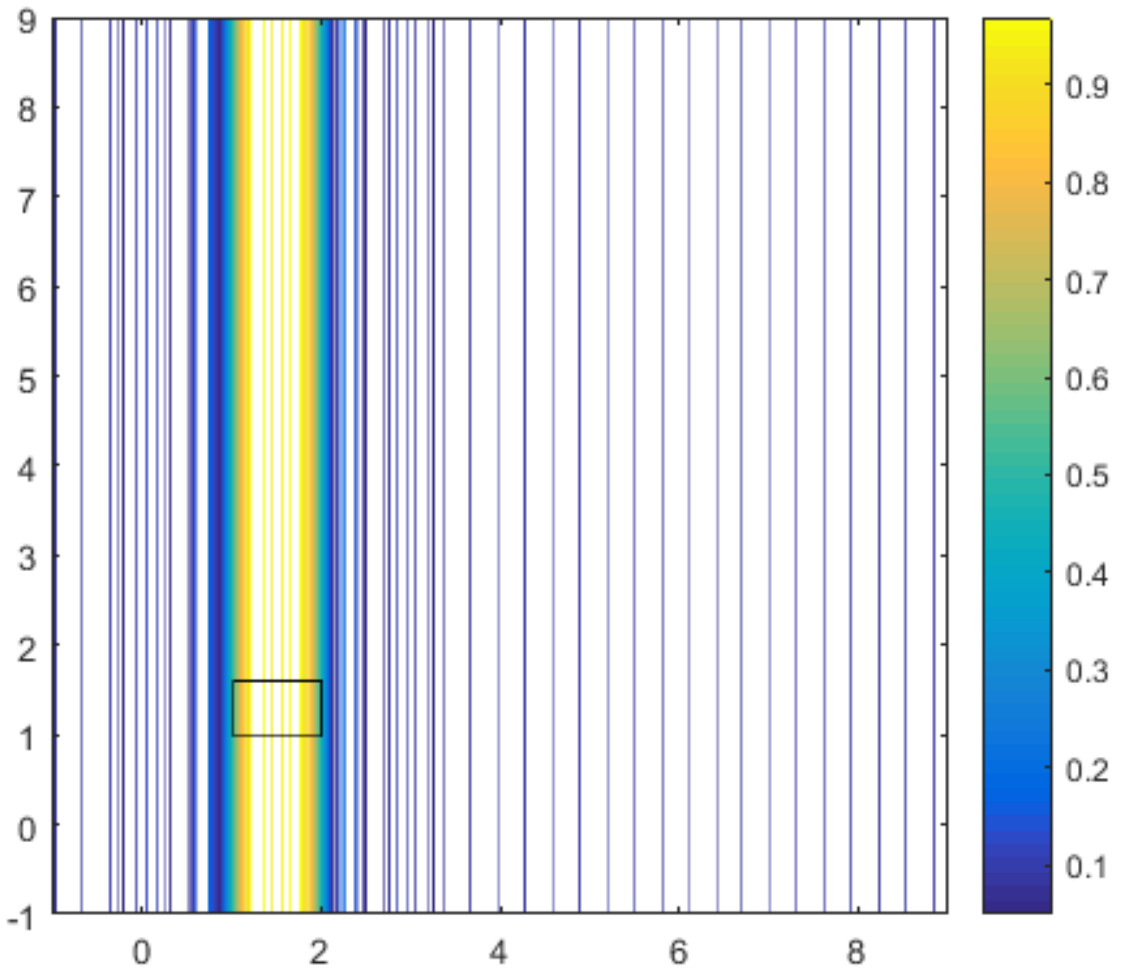}}
  \subfigure[\textbf{$z_0=(4,4)$.}]{
    \includegraphics[width=2in]{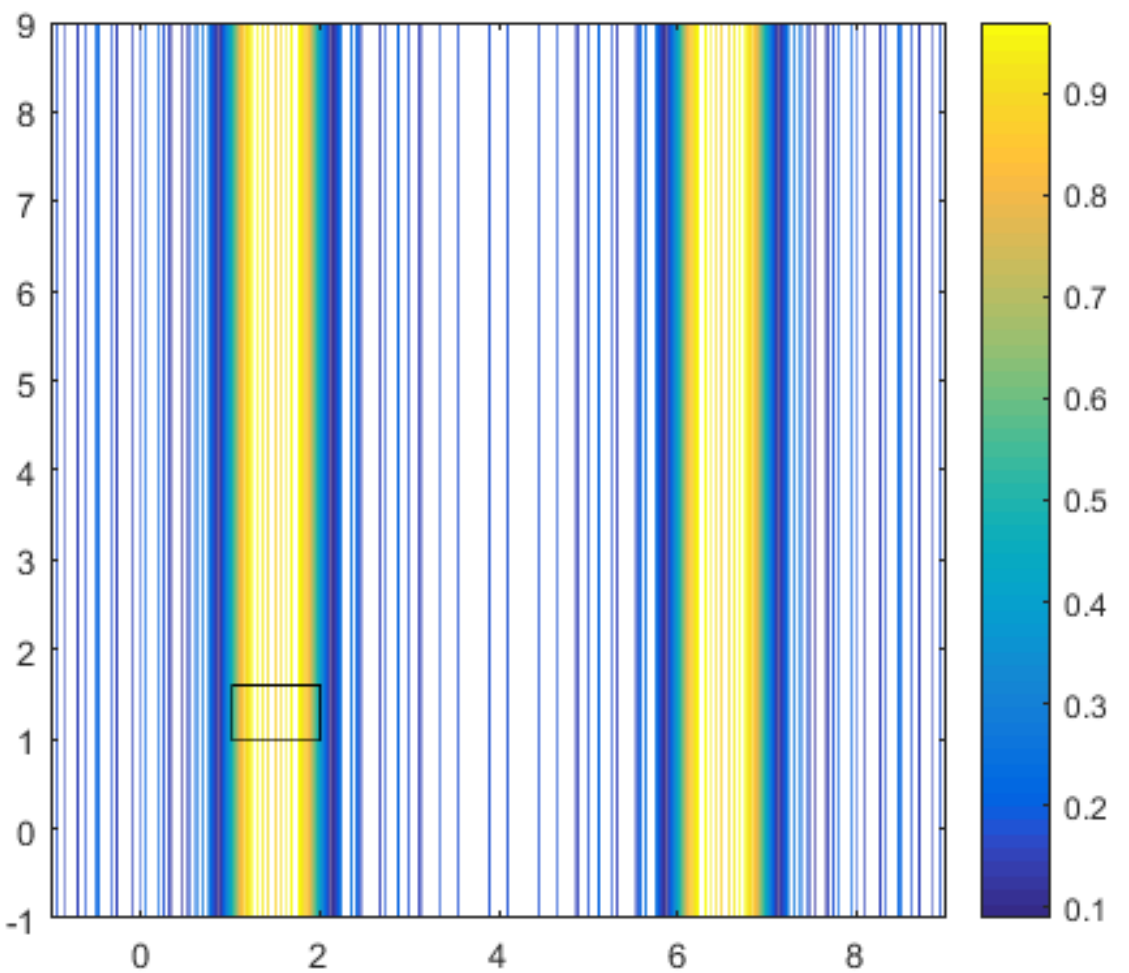}}
  \subfigure[\textbf{$z_0=(12,12)$.}]{
    \includegraphics[width=2in]{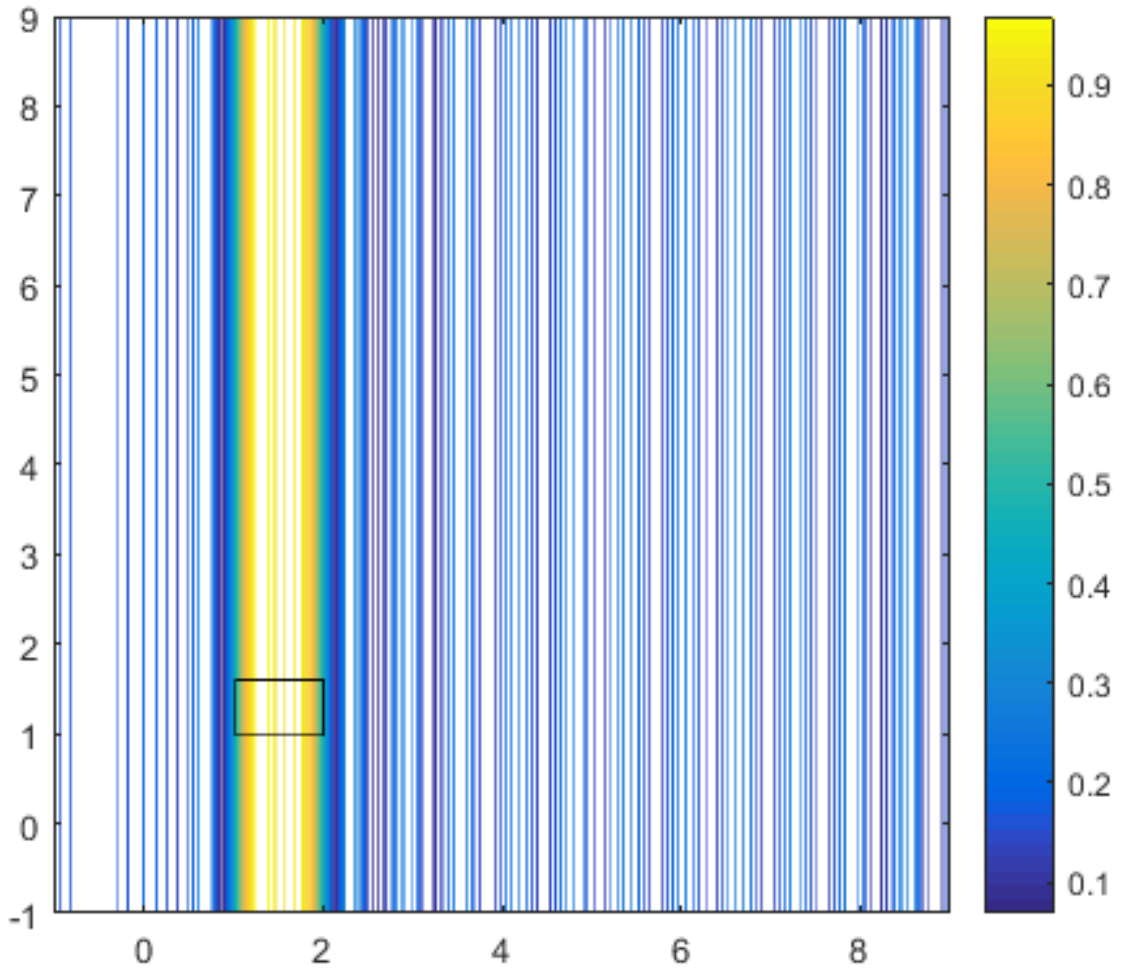}}
\caption{${\bf I^{\Theta}_{z_0}}$ with one observation direction when $S = 5$.}
\label{rec1}
\end{figure}

Next we consider two observation angles $0$ and $\pi/2$, we plot the indicators in Fig. \ref{rec2}.  Since the observation directions are perpendicular to each other, the strips are perpendicular to each other too. The source support must be located in the cross sections of the two strips. To find the correct source support, one may consider choosing a reference point far away from the sampling domain (as shown in Figure \ref{rec2}(c)), or using multiple observation directions (as shown in Figure \ref{rec3}).

\begin{figure}[htbp]
  \centering
  \subfigure[\textbf{$z_0=(1.5,4)$.}]{
    \includegraphics[width=2in]{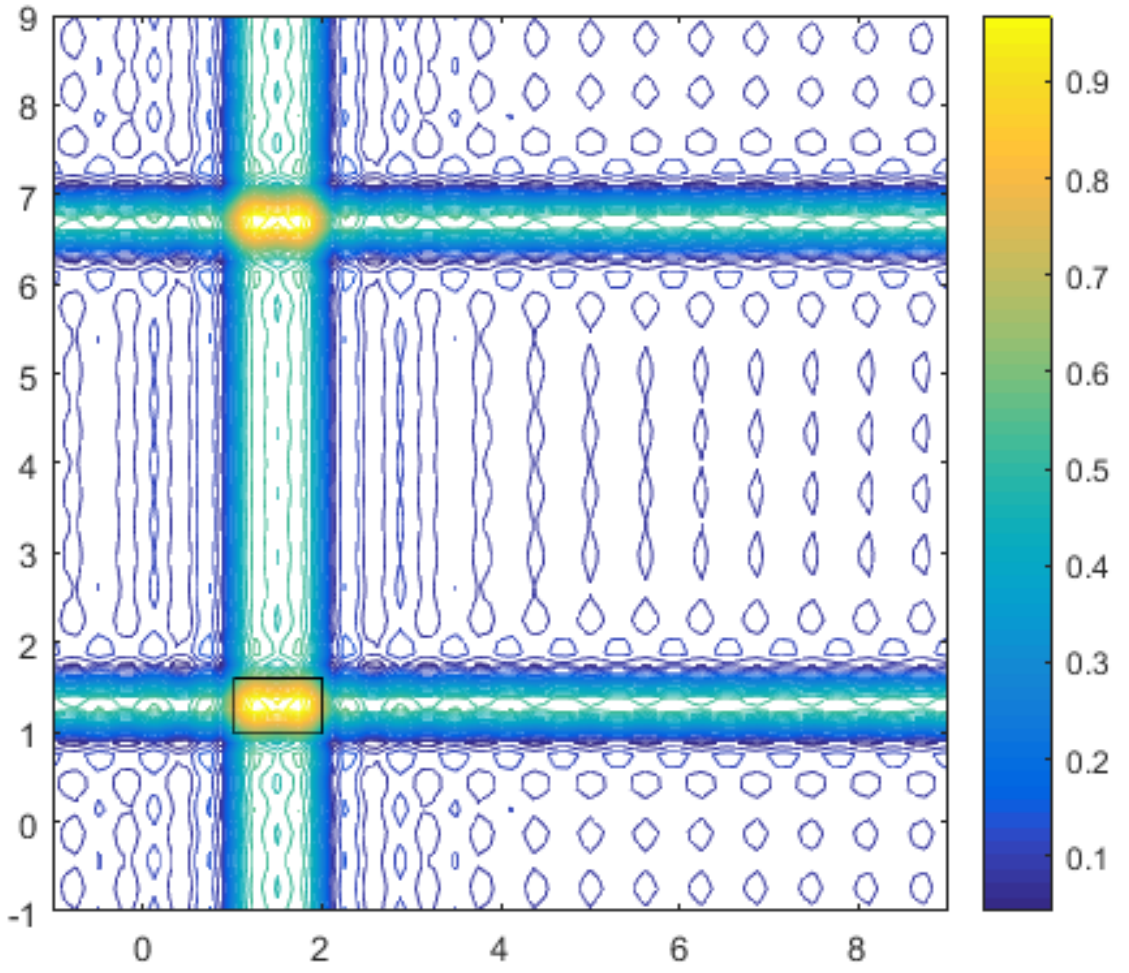}}
  \subfigure[\textbf{$z_0=(4,4)$.}]{
    \includegraphics[width=2in]{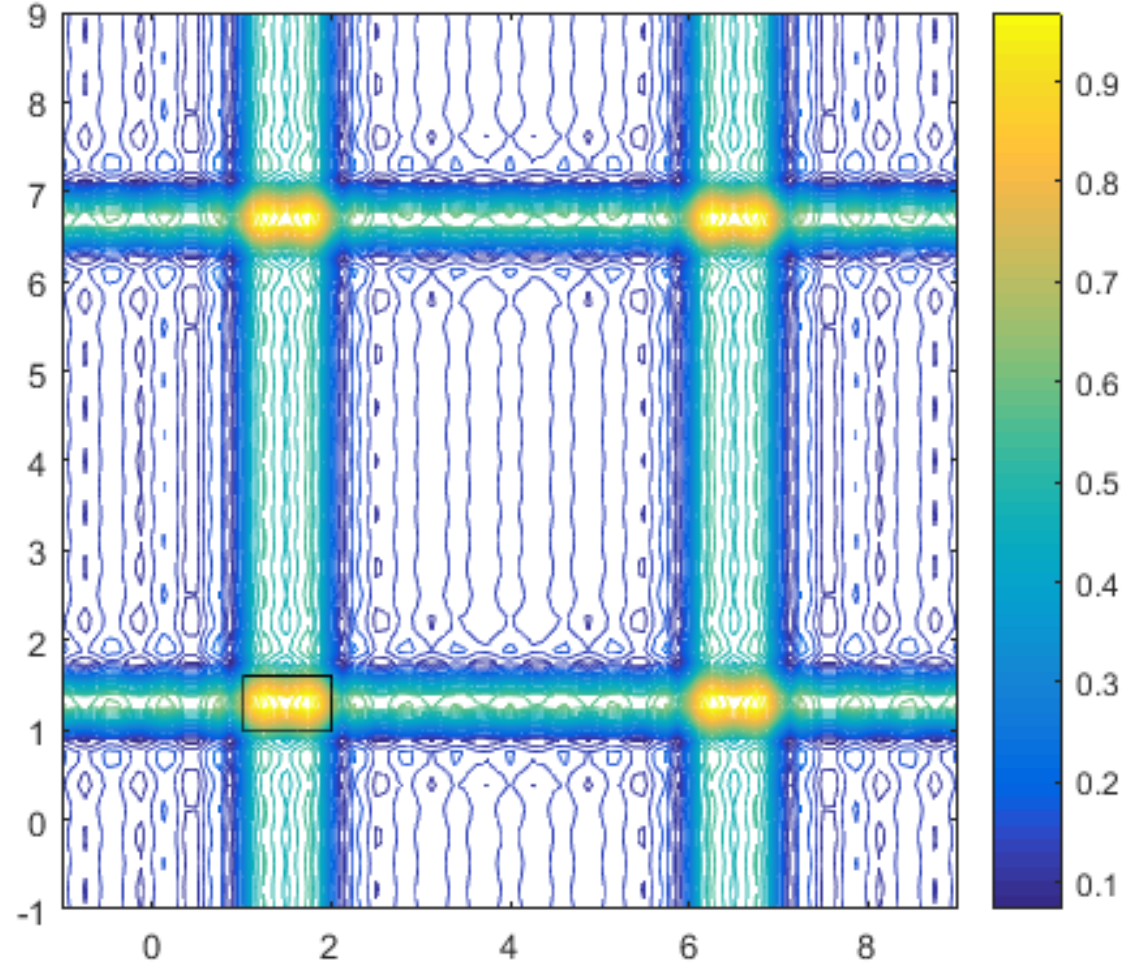}}
  \subfigure[\textbf{$z_0=(12,12)$.}]{
    \includegraphics[width=2in]{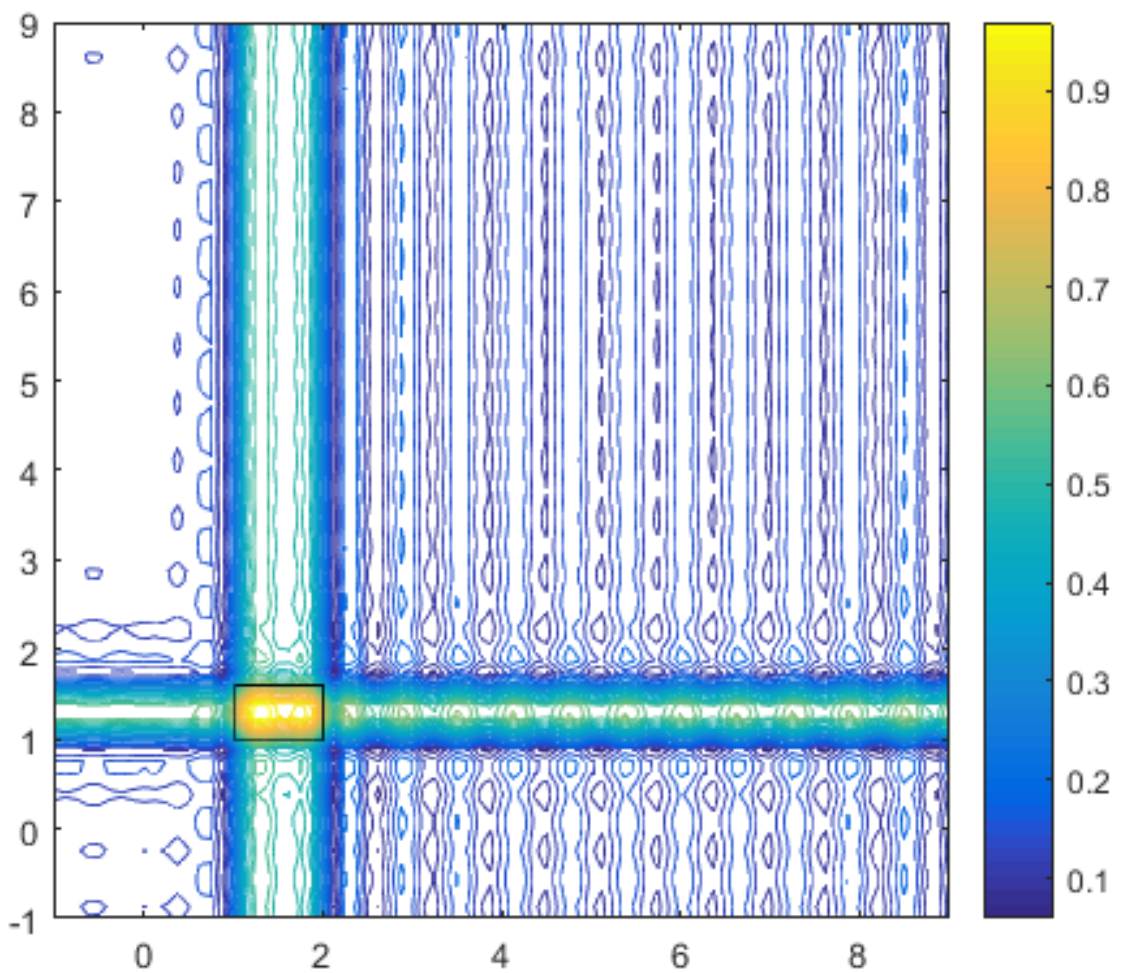}}
\caption{${\bf I^{\Theta}_{z_0}}$ with two observation directions when $S = 5$.}
\label{rec2}
\end{figure}

\subsection{${\bf I^{\Theta}_{z_0}}$ with multiple observation directions}
We consider three sources in this part,  the first one is still the rectangle given by $(1, 2) \times(1,1.6)$ with $S = 5$, the second one is a combination of a
a rectangle given by $(1, 1.6) \times(1,1.4)$ and  a disc with radius
0.2 centered at $(-0.5,-0.5)$ with $S=(x^2-y^2+5)k$, the third one is a L-shaped domain given by  $(0,2) \times(0,2) \setminus (1/16, 2) \times(1/16,2)$ with $S = 5$.

Now we use 20 observation angles $\hx_j, j=1,\cdots, 20$ such that  $\hx_j=-\pi/2+j\pi/20$. Note that $\hx_j\in(-\pi/2,\pi/2)$. Fig. \ref{rec3} gives the results for
rectangle with different $z_0$. The locations and sizes
of support of $S$ are reconstructed correctly.
\begin{figure}[htbp]
  \centering
  \subfigure[\textbf{$z_0=(1.5,4)$.}]{
    \includegraphics[width=2in]{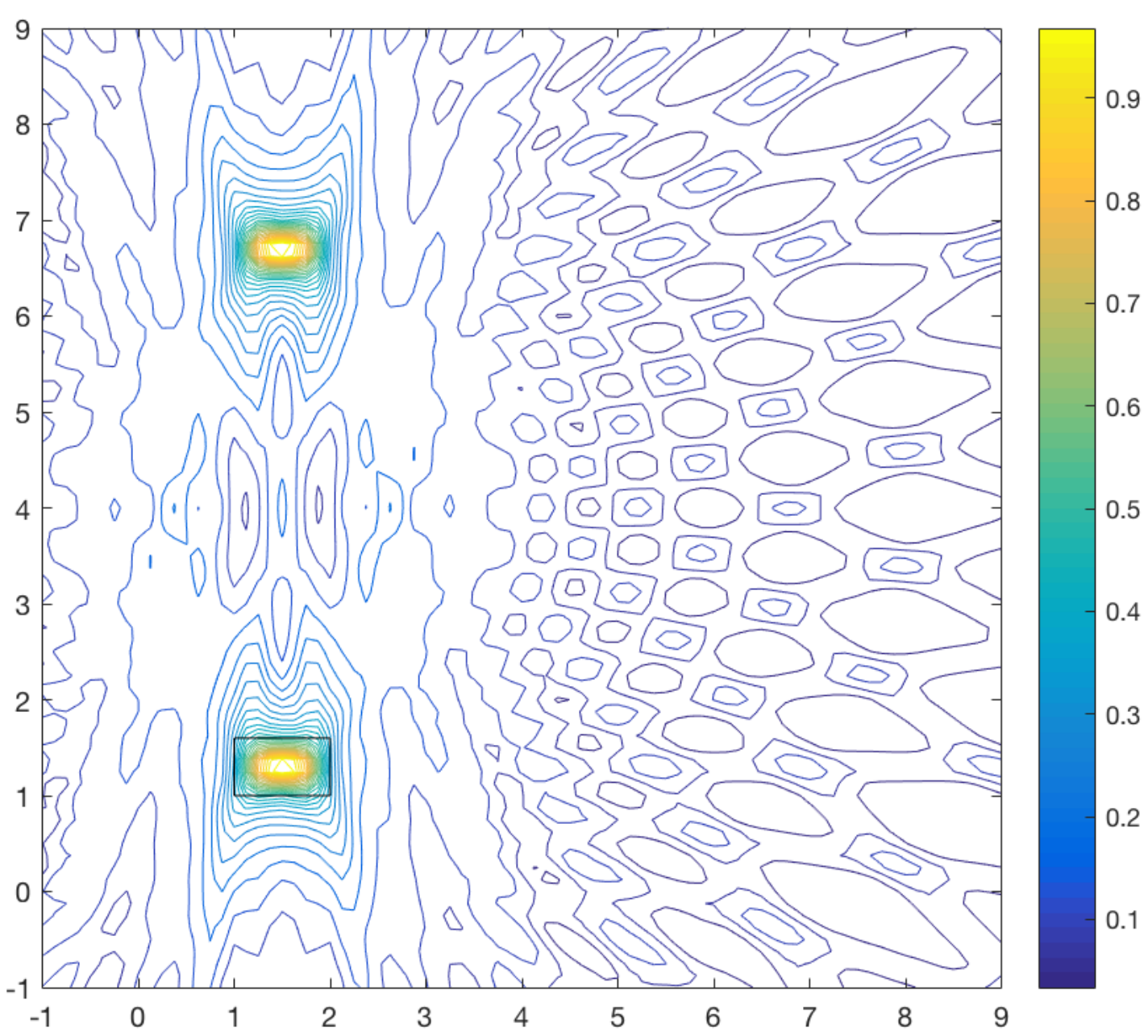}}
  \subfigure[\textbf{$z_0=(4,4)$.}]{
    \includegraphics[width=2in]{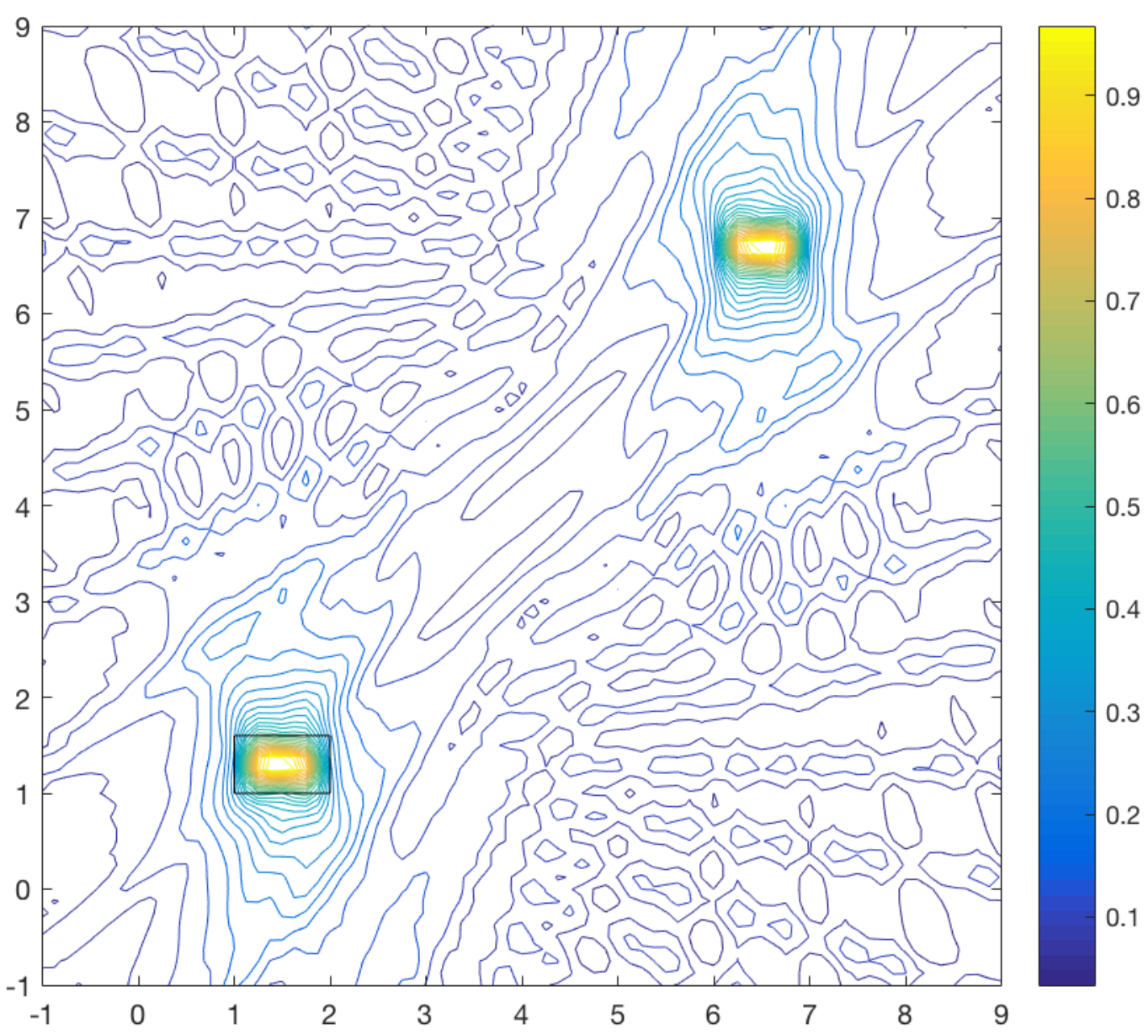}}
  \subfigure[\textbf{$z_0=(12,12)$.}]{
    \includegraphics[width=1.98in]{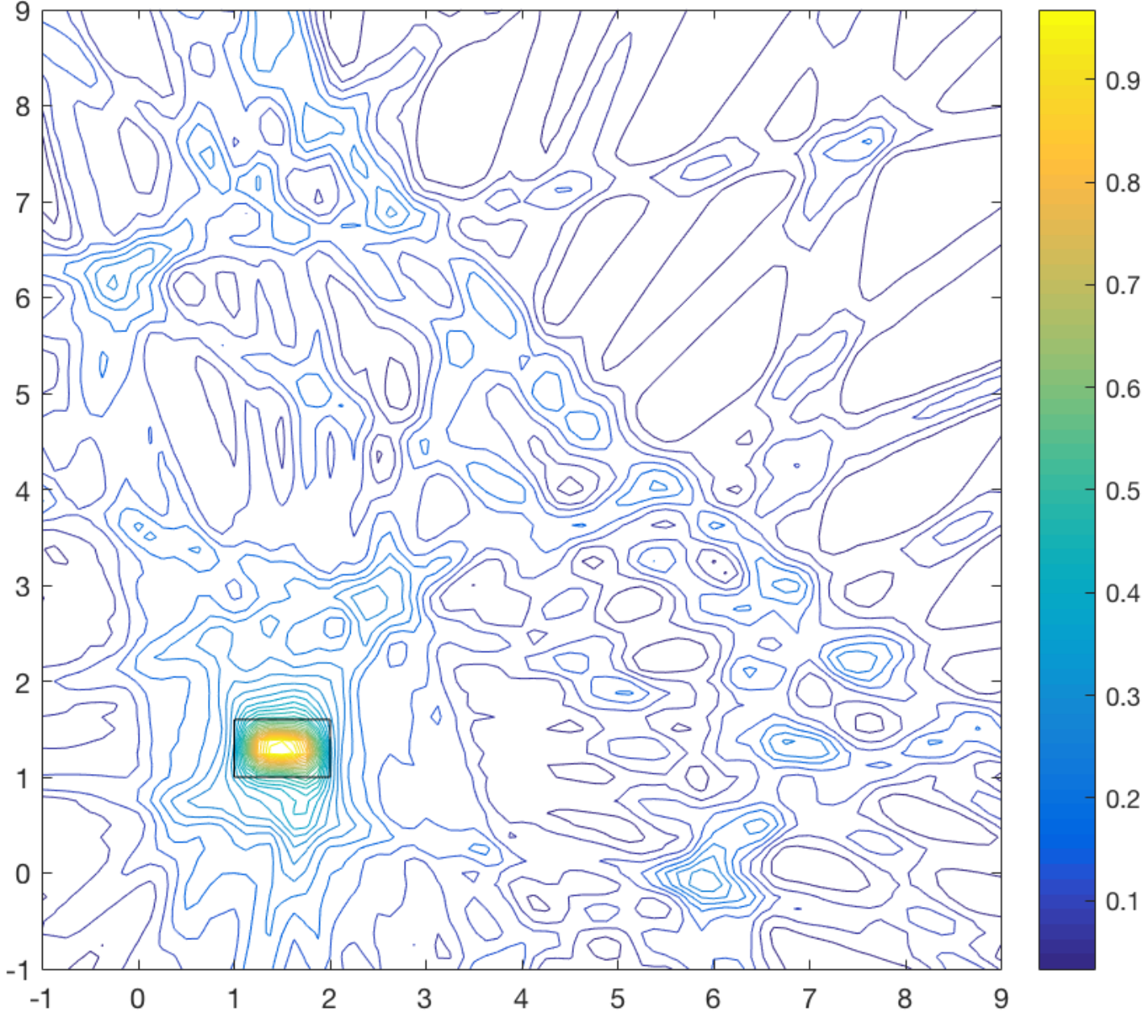}}
\caption{${\bf I^{\Theta}_{z_0}}$ with multiple directions for rectangle when $S = 5$.}
\label{rec3}
\end{figure}

Figure \ref{comb}(a) gives the results for the two source supports with $z_0=(12,12)$.  Figure \ref{Lshaped}(b) gives the results for
the L-shaped domain with $z_0=(12,12)$.

\begin{figure}[htbp]
  \centering
  \subfigure[\textbf{$S=(x^2-y^2+5)k$.}]{
    \includegraphics[width=2in]{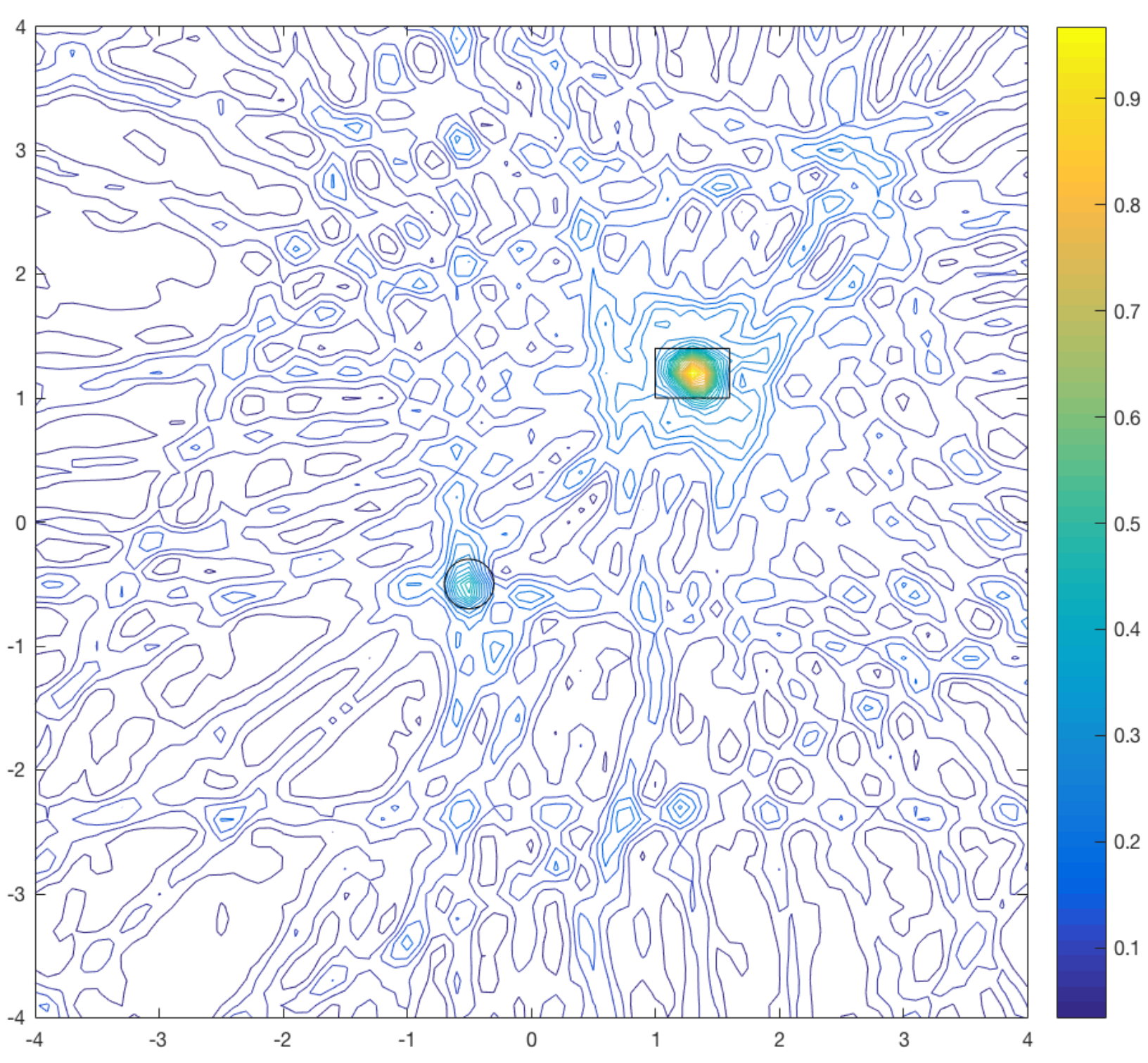}}
  \subfigure[\textbf{$S = 5$.}]{
    \includegraphics[width=2in]{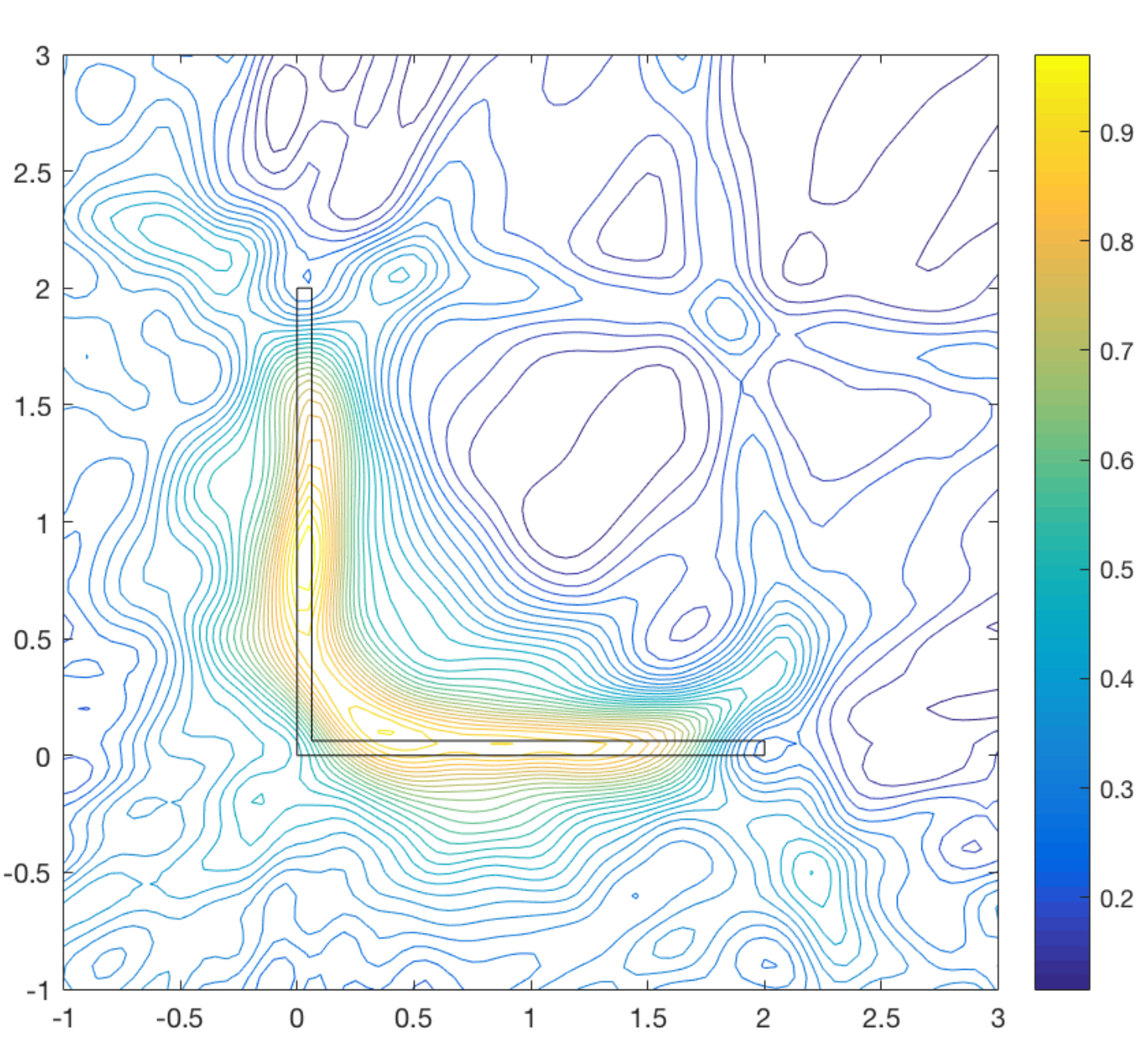}}
\caption{${\bf I^{\Theta}_{z_0}}$ with multiple directions with $z_0=(12,12)$. (a) two sources; (b) $L$-shaped domain.}
\label{comb}
\end{figure}

\subsection{The validity of the phase retrieval scheme}

This example is designed to check the phase retrieval scheme
proposed in the previous section. The underlying scatterer is still the rectangle given by $(1, 2) \times(1,1.6)$ with $S = 5$.
In Fig. \ref{error1} and Fig. \ref{error2}, we compare the phase retrieval data with the exact one, the real part of far field
pattern at a fixed direction $\hx=(1,0)$ is given. We observe that our phase retrieval scheme is very robust to noise.

\begin{figure}[htbp]
  \centering
  \subfigure[\textbf{$10\%$ noise.}]{
    \includegraphics[height=1.3in,width=2in]{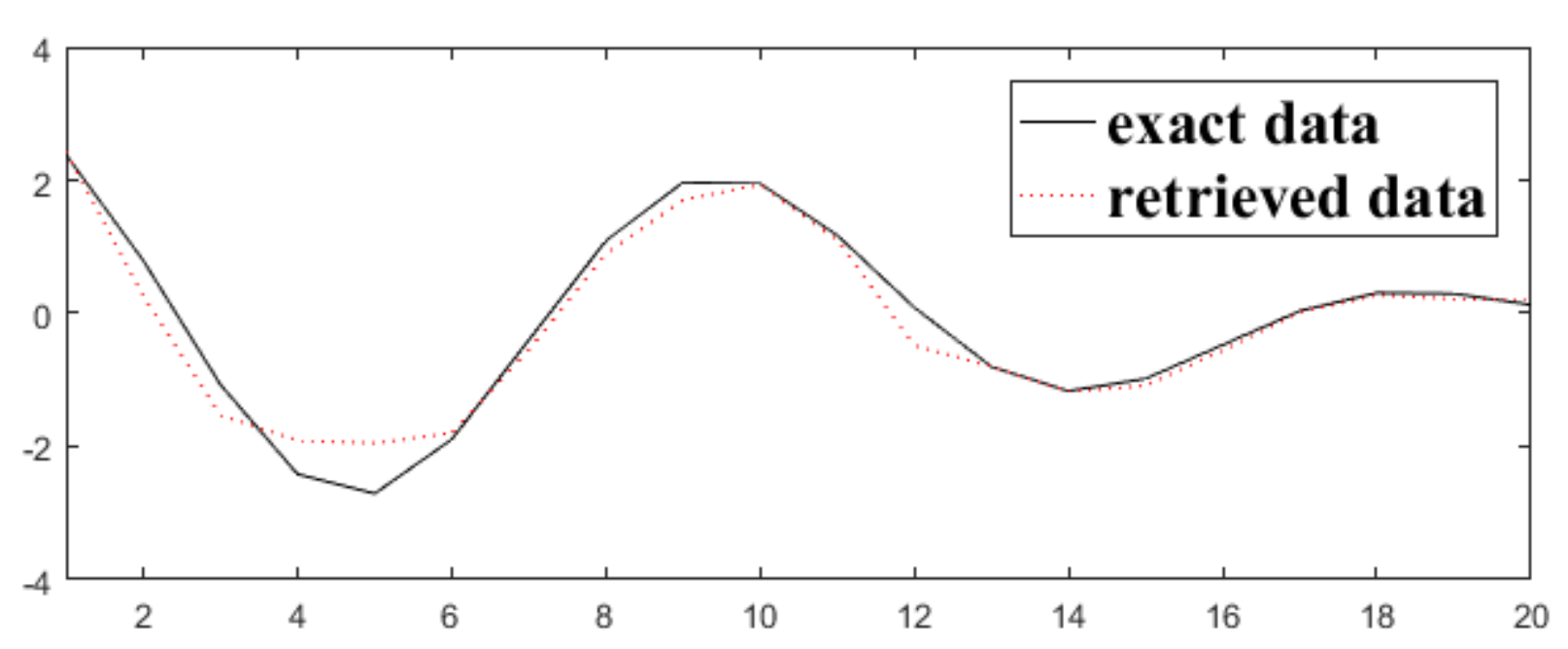}}
  \subfigure[\textbf{$20\%$ noise.}]{
    \includegraphics[height=1.3in,width=2in]{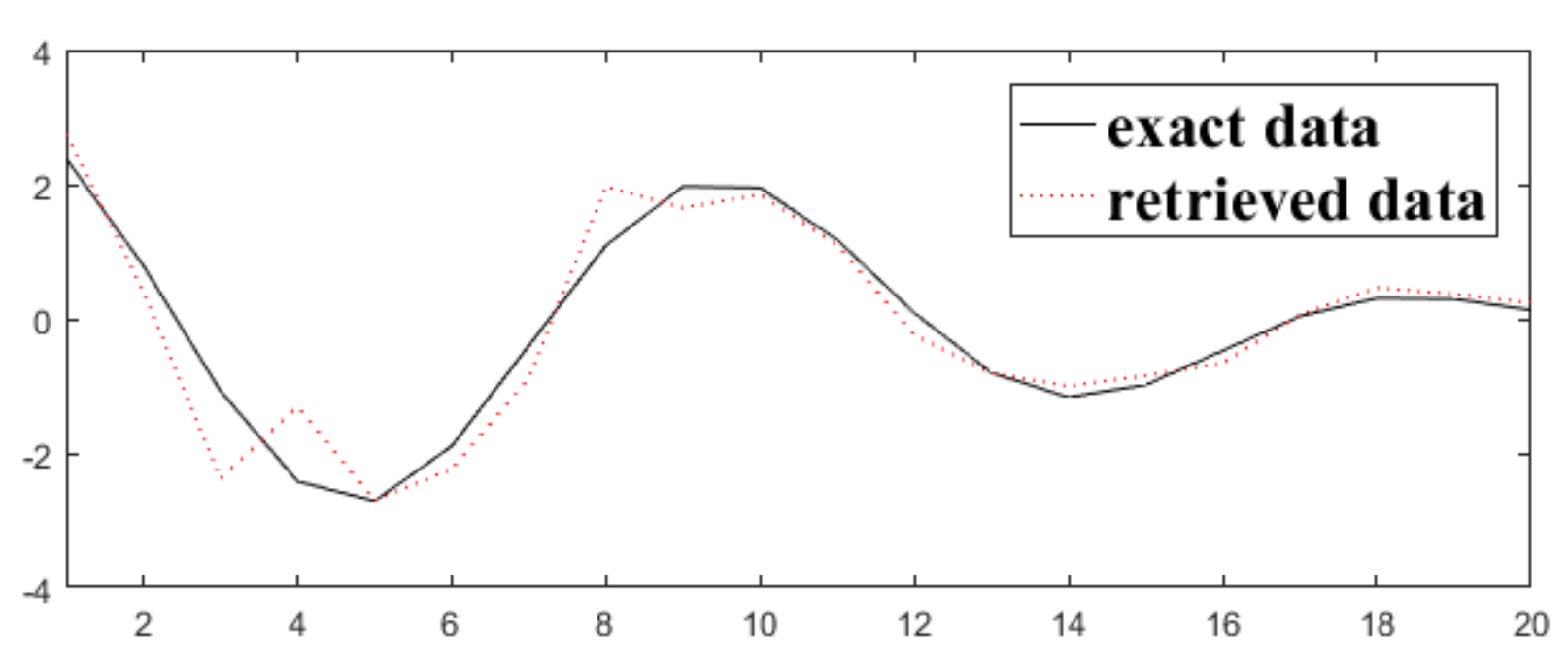}}
  \subfigure[\textbf{$30\%$ noise.}]{
    \includegraphics[height=1.3in,width=2in]{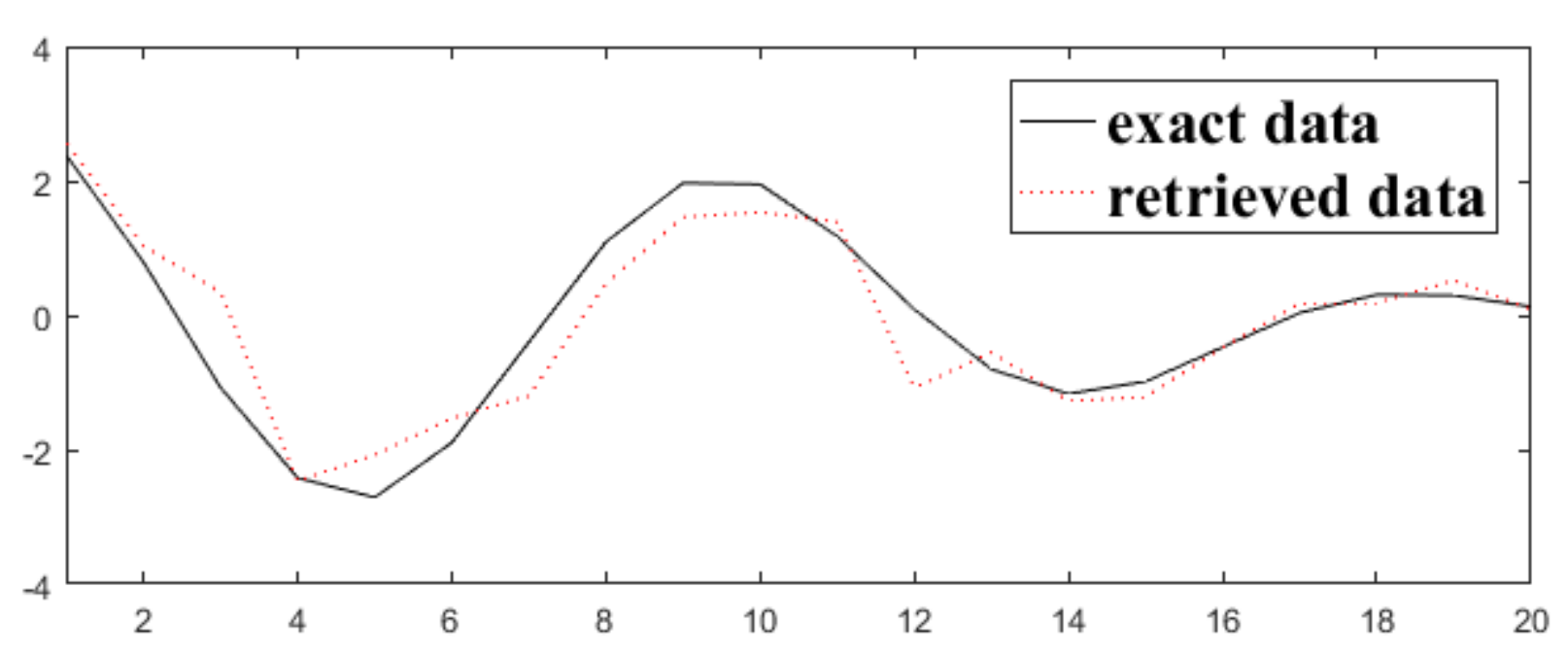}}
\caption{{\bf Example PhaseRetrieval.}\, Phase retrieval for the real part of the far field pattern with relative error at a fixed direction $\hx=(1,0)$.}
\label{error1}
\end{figure}

\begin{figure}[htbp]
  \centering
  \subfigure[\textbf{0.1 noise.}]{
    \includegraphics[height=1.3in,width=2in]{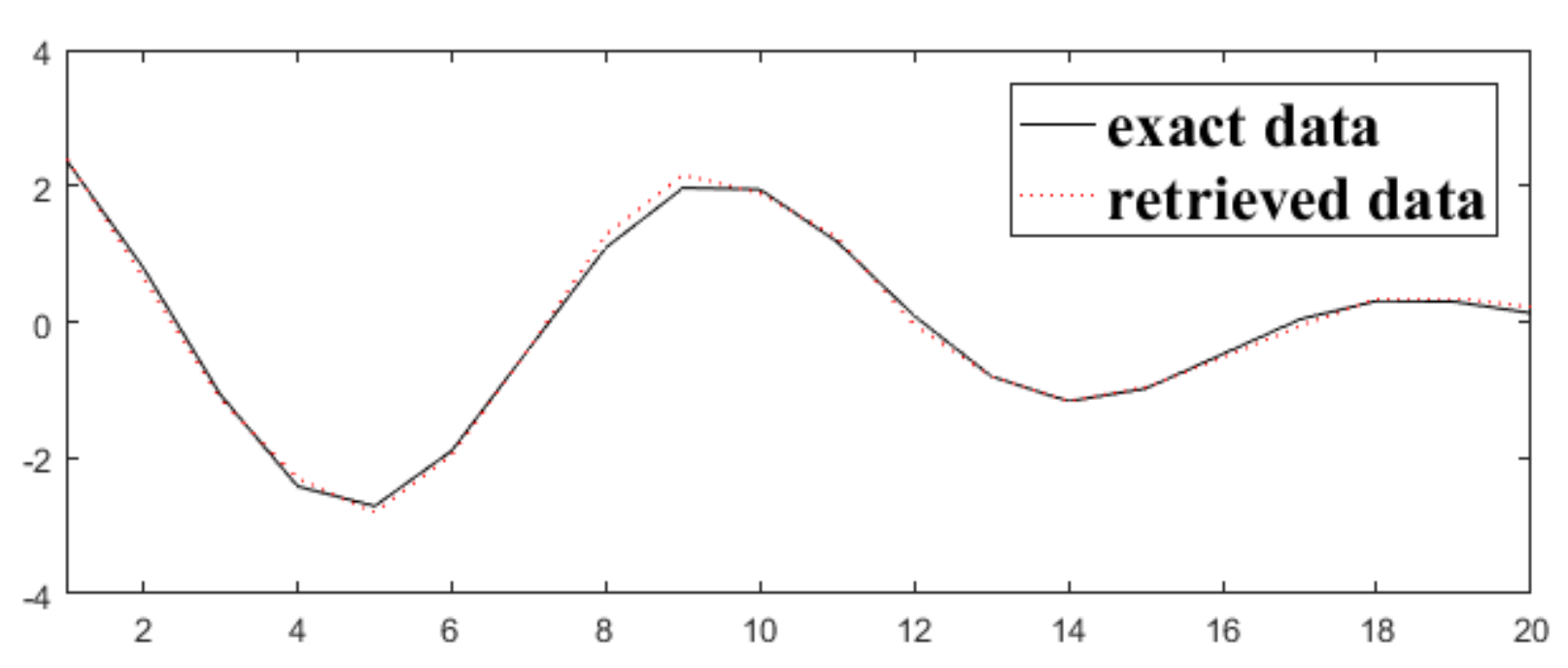}}
  \subfigure[\textbf{0.3 noise.}]{
    \includegraphics[height=1.3in,width=2in]{pic/Errrectabsolute01.pdf}}
  \subfigure[\textbf{0.5 noise.}]{
    \includegraphics[height=1.3in,width=2in]{pic/Errrectabsolute01.pdf}}
\caption{{\bf Example PhaseRetrieval.}\, Phase retrieval for the real part of the far field pattern with absolute error at a fixed direction $\hx=(1,0)$.}
\label{error2}
\end{figure}

\subsection{${\bf I_2}(z)$ with broadband sparse data}
In this part, we use $z_0=(4,4)$ and $\tau=\pm 1, i$. We first use the {\bf Phase Retrieval Scheme} to obtain the phased data, and then reconstruct the source support by the indicator $I_2$.

Figure \ref{PRonetwo} shows the reconstructions of the rectangle given by $(1, 2) \times(1,1.6)$ with one or two observation directions.
Different to Figures \ref{rec2}(b) and \ref{rec3}(b), the false strips disappear now.
We also consider sources with extended supports, the first one is an equilateral triangle with vertices $(-2,0), (1,0),(-1/2,3/2\sqrt{3})$, and the second one is a thin slab given
by $(-2,2)\times (0,0.1)$. Figure \ref{large} shows that both the triangle and the thin slab are reconstructed very well, even $10\%$ noise is considered.

\begin{figure}[htbp]
   \centering
   \subfigure[\textbf{$\hx=(1,0)$.}]{
     \includegraphics[width=2.2in]{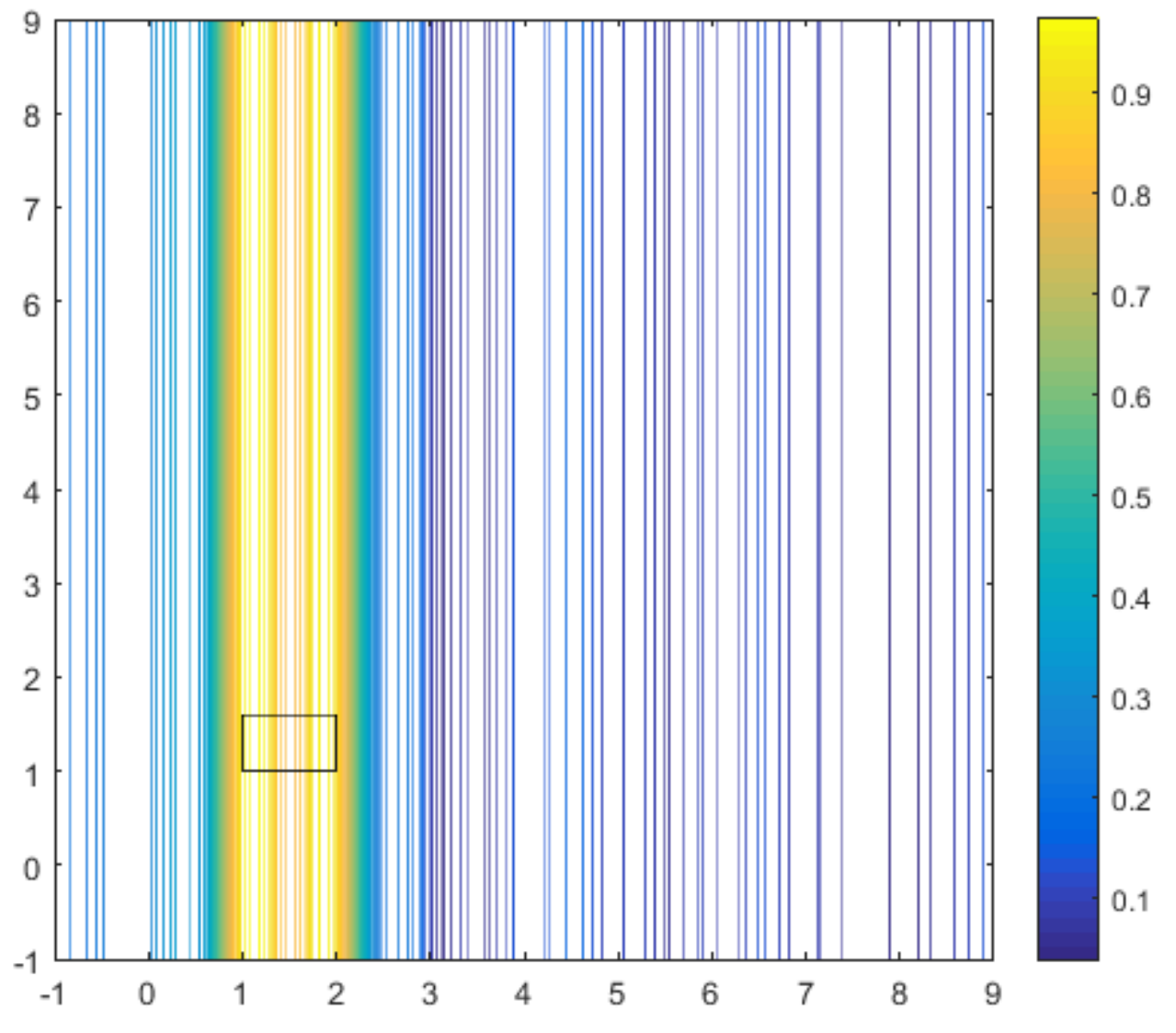}}
   \subfigure[\textbf{$\hx= (1,0), (0,1)$.}]{
     \includegraphics[width=2.2in]{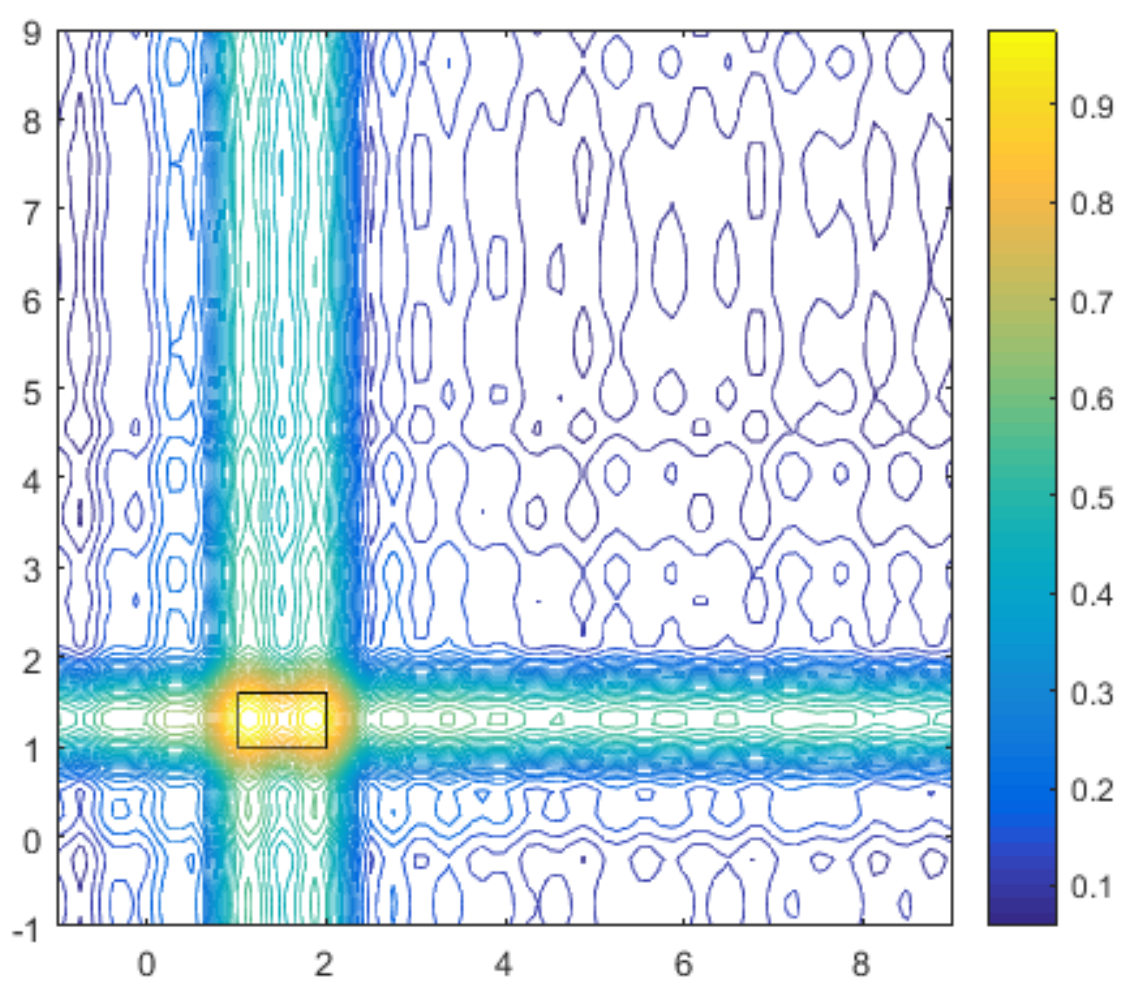}}
 \caption{Reconstructions of the rectangle with one and two observation directions.}
 \label{PRonetwo}
\end{figure}

\begin{figure}[htbp]
  \centering
  \subfigure{
    \includegraphics[height=2in,width=2in]{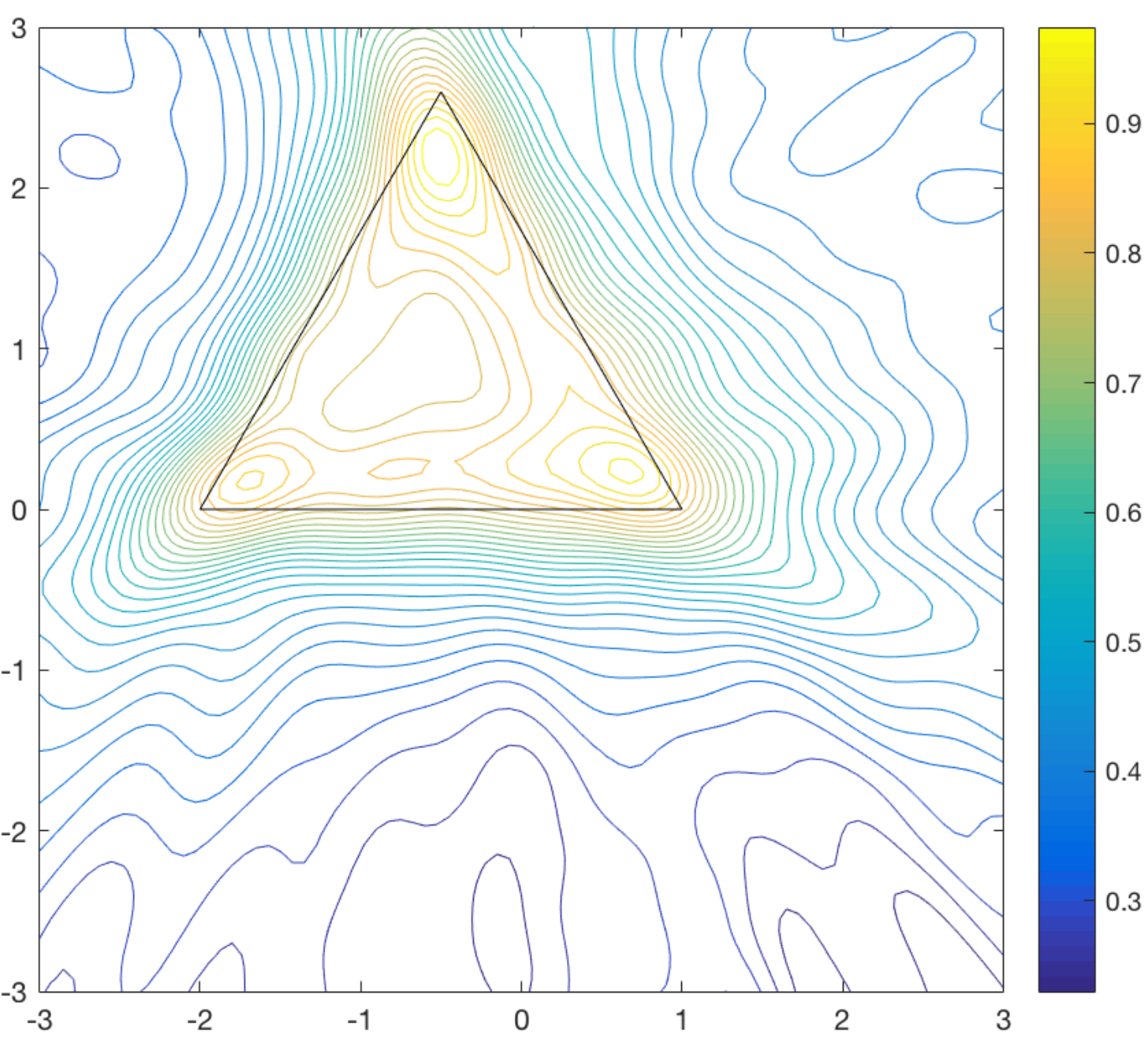}}
  \subfigure{
    \includegraphics[height=2in,width=2in]{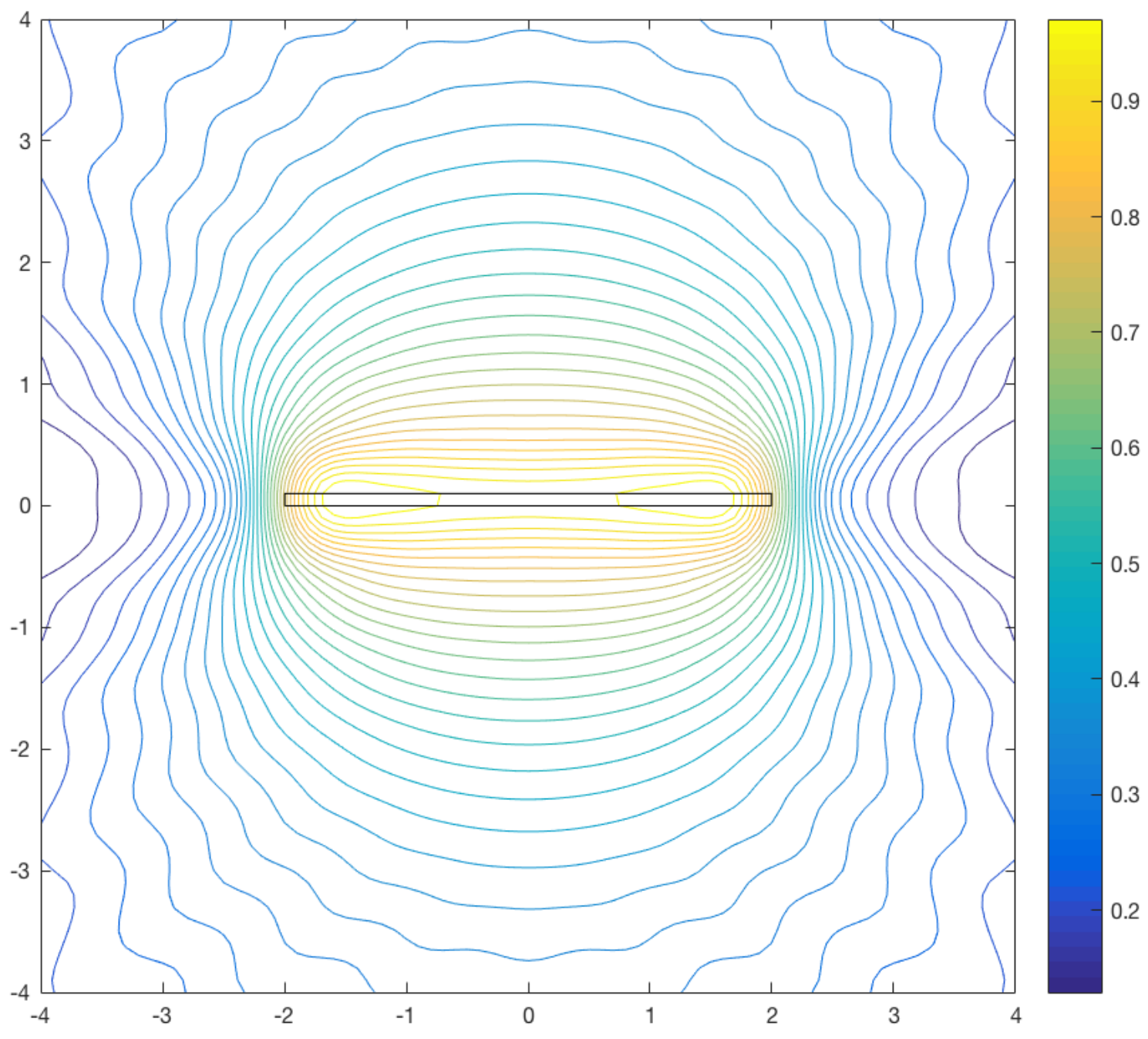}}
\caption{Reconstructions of extended objects with $S=5$ and $10\%$.}
\label{large}
\end{figure}

\section*{Acknowledgement}
The research of X. Ji is partially supported by the NNSF of China with Grant Nos. 11271018 and 91630313,
and National Centre for Mathematics and Interdisciplinary Sciences, CAS.
The research of X. Liu is supported by the NNSF of China under grant 11571355 and the Youth Innovation Promotion Association, CAS.
The research of B. Zhang is partially supported by the NNSF of China under grant 91630309.

\bibliographystyle{SIAM}

\end{document}